\documentclass[10-pt]{article}
\usepackage{fancyhdr}
\usepackage[british]{babel}

\usepackage[letterpaper, top=4cm,bottom=3.7cm,left=4cm,right=4cm,marginparwidth=2cm]{geometry}

\usepackage{amsmath}
\usepackage{amssymb}
\usepackage{amsthm}
\usepackage{graphicx}
\usepackage{xspace}
\usepackage[T1]{fontenc}
\usepackage[utf8]{inputenc}
\usepackage{float}
\usepackage{caption}
\usepackage{quiver}
\usepackage{yhmath}
\usepackage[nottoc,numbib]{tocbibind}
\usepackage{scalerel}
\usepackage{footmisc}
\usepackage[useregional]{datetime2}
\DTMlangsetup[en-GB]{showdayofmonth=false}

\usepackage[maxbibnames=99, style=numeric,backend=biber, url=false, doi=false, isbn=false, date=year]{biblatex}

\usepackage{filecontents}

\newtheorem{proposition}{Proposition}[section]
\newtheorem{theorem}[proposition]{Theorem}
\newtheorem{lemma}[proposition]{Lemma}
\newtheorem{example}[proposition]{Example}
\newtheorem{definition}[proposition]{Definition}
\newtheorem{remark}[proposition]{Remark}

\newtheorem{corollary}[proposition]{Corollary}
\newtheorem{question}[proposition]{Question}
\newtheorem{radonsetup}[proposition]{Radon setup}
\newtheorem{facts}[proposition]{Facts}

\newtheorem{conjecture}[proposition]{Conjecture}

\newtheorem{terminology}[proposition]{Terminology}

\newcommand\doverline[1]{\ThisStyle{%
  \setbox0=\hbox{$\SavedStyle\overline{#1}$}%
  \ht0=\dimexpr\ht0-.15ex\relax
  \overline{\copy0}%
}}
\newcommand*{\sqihat}{\skew{15}{\hat}{\sqrt{I}}}
\newcommand{\congto}{\xrightarrow{\raisebox{-0.5ex}[0ex][0ex]{$\sim$}}}
\makeatletter
\newcommand{\etale}{\'etal\@ifstar{\'e}{e\xspace}}
\makeatother
\newcommand\blfootnote[1]{%
  \begingroup
  \renewcommand\thefootnote{}\footnote{#1}%
  \addtocounter{footnote}{-1}%
  \endgroup
}

\newcommand{\Addresses}{{
  \bigskip
\noindent\textsc{Department of Mathematics, Massachusetts Institute of Technology, Cambridge, USA,}\par\nopagebreak
\noindent Email: \texttt{tong.g.h.zhou@gmail.com}
  }}

\usepackage[colorlinks=true, allcolors=blue]{hyperref}

\newcommand{\p}{\mathbf{p}}
\newcommand{\q}{\mathbf{q}}

\title{On the stability of vanishing cycles of \etale sheaves in positive characteristic}
\author{Tong Zhou}
\date{}

\begin{document}
\maketitle
\begin{abstract}\blfootnote{February 2026}
In positive characteristic, in contrast to the complex analytic case, vanishing cycles are highly sensitive to test functions (the maps to the henselian traits). We study this dependence and show that on a smooth surface, this dependence is generically only up to a finite jet of the test functions. We conjecture that this continues to hold in higher dimensions. We also study the class of sheaves whose vanishing cycles have the strongest stability. Among other things, we show that tame simple normal crossing sheaves belong to this class, and this class is stable under the Radon transform.
\end{abstract}

\tableofcontents

\section{Introduction}
The microlocal point of view of studying objects on a manifold via constructions on the cotangent bundle was introduced by Mikio Sato in the field of partial differential equations. This idea led to the birth of microlocal analysis and spread out to other fields of mathematics. In \cite{kashiwara_sheaves_1990}, Masaki Kashiwara and Pierre Schapira systematically developed the theory of sheaves on real and complex analytic manifolds from this point of view.\\

Question: what does the theory look like for $\ell$-adic sheaves on schemes in positive characteristic? In this context, among many other works, we point out: Jean-Louis Verdier defined the specialisation (\cite{verdier_specialisation_1983}), which, however, kills wild ramifications, Ahmed Abbes and Takeshi Saito defined the characteristic class (\cite{abbes_characteristic_2007}), and considered the microlocal analysis for \etale sheaves in a local version which goes beyond tame ramifications (\cite{abbes_analyse_2009}), and Kazuya Kato and Saito defined the Swan class (\cite{kato_ramification_2008}). A recent breakthrough is from Alexander Beilinson (\cite{beilinson_constructible_2016}) who successfully defined the singular support ($SS$), and Saito (\cite{saito_characteristic_2017}) defined the characteristic cycle ($CC$) based on that.\\

The starting point of this paper is the following: apart from the $SS$ and the $CC$, another key notion in microlocal sheaf theory is the microstalk. The microstalk plays the same role in microlocal sheaf theory as the stalk in usual sheaf theory. Is there a similar notion in the positive characteristic algebraic context? In the complex analytic context, one definition of the microstalk is via the vanishing cycles functor. We introduce the following notions before discussing further. Let $\mathcal{F}$ be a sheaf (see Conventions) on a complex analytic manifold $X$.

\begin{definition}[transverse test function (analytic)]\label{ttfun}
    A \underline{transverse test function (ttfun)} of $\mathcal{F}$ at a \emph{smooth} point $(x, \xi)$ of $SS\mathcal{F}$ is a complex analytic function $f$ defined on an open neighbourhood U of $x$ such that the following are satisfied:\\
(i) $f(x)=0$;\\
(ii) the graph $\Gamma_{df}$ of the differential of $f$ intersects $SS(
\mathcal{F})$ only at $(x, \xi)$, and the intersection is transverse.
\end{definition}

\begin{definition}[transverse test family (analytic)]\label{Cttfamdef}
    A \underline{transverse test family (ttfam)} of $\mathcal{F}$ at a \emph{smooth} point $(x, \xi)$ of $SS\mathcal{F}$, denoted by $(T,U,V,f)$, is the following data (here $\mathbb{A}^1$ denotes $\mathbb{C}$ viewed as a complex analytic manifold):\\
\[\begin{tikzcd}
	{} & {U\times T} & V & {x_T:=x\times T} & {} \\
	&& {\mathbb{A}^1_T:=\mathbb{A}^1\times T} \\
	&& {}
	\arrow[hook', from=1-3, to=1-2]
	\arrow[hook', from=1-4, to=1-3]
	\arrow["f", from=1-3, to=2-3]
\end{tikzcd}\]
where:\\
(i) T is a connected complex analytic manifold, serving as the parameter space of the family. We will often identify T with $0\times T\subseteq\mathbb{A}^1\times T$.\\
(ii) U is an open neighbourhood of x, V is an open of $U\times T$ containing $x_T$.\\
(iii) f is a complex analytic map such that, for all $s\in T$, the \underline{s-slice} $f_s: V_s$ $(:= V\times_{\mathbb{A}^1_T}\mathbb{A}^1_s)\rightarrow \mathbb{A}^1_s$ is a ttfun with respect to $\mathcal{F}$ at $(x,\xi)$.
\end{definition}

\begin{remark}
    Later in the introduction, we will also use the notions of ttfun and ttfam in the algebraic context, which are the obvious analogues of their analytic counterparts. See Definition \ref{ttfun-alg} and Definition \ref{ttfamdef} for the precise definitions.
\end{remark}

Now, for $(x,\xi)$ a smooth point of $SS\mathcal{F}$, the \underline{microstalk} of $\mathcal{F}$ at $(x,\xi)$ is defined to be $\phi_f(\mathcal{F})_x$, where $f$ is any transverse test function at $(x,\xi)$. It is well-defined because of the crucial fact that, in the complex analytic context, vanishing cycles have strong stability with respect to the variation of the ttfun. More precisely:

\begin{theorem}[{\cite[\nopp 7.2.4]{kashiwara_microlocal_1985}}, Theorem \ref{phistability/C}]
Let $X$ be a complex analytic manifold, $\mathcal{F}\in D(X)$, and $(x, \xi)$ a smooth point of $ SS\mathcal{F}$. Then:\\
(i) For every two ttfun's $f$ and $g$ of $\mathcal{F}$ at $(x, \xi)$, there exists (noncanonically) an isomorphism $\phi_f(\mathcal{F})_x$ $\cong \phi_g(\mathcal{F})_x$ in $D^b_c(\mathbb{C}[\mathbb{Z}])$.\\
(ii) For every ttfam $(T,U,V,f)$ of $\mathcal{F}$ at $(x, \xi)$, $\phi_{pf}(\mathcal{F}_V)$ is a local system on $x_T$, with the stalk at $x\times s$ canonically isomorphic to $\phi_{f_s}(\mathcal{F})_x$, for all $s\in T$ $(=0\times T\subseteq \mathbb{A}^1\times T)$. Here $p$ is the projection $\mathbb{A}^1\times T\rightarrow\mathbb{A}^1$, $\mathcal{F}_V$ is the pullback of $\mathcal{F}$ to $V$.
\end{theorem}

Statement (i) says that the vanishing cycles, as vector spaces with monodromy actions, are independent of the choice of the ttfun. The stronger statement (ii) is a family version of (i).\\

These fail completely in the positive characteristic algebraic context because of wild ramifications. Here is an example (see §\ref{sectiononfailure} for details): consider $\mathbb{A}^2$ over an algebraically closed field of characteristic $p>3$. Let $D$ be the $y$-axis and $U$ be the complement. Let $\mathcal{F}$ be the !-extension to $\mathbb{A}^2$ of the Artin-Schreier sheaf on $U$ determined by the equation $t^p-t=y/x^p$ and a nontrivial character of its covering group $\mathbb{Z}/p$. One can show $SS\mathcal{F}=T^*_XX \cup \langle dy\rangle_D$, where $\langle dy\rangle_D$ denotes the subspace of $D\times_X T^*X$ consisting of covectors proportional to $dy$. Consider the vanishing cycles with respect to the following two functions: $f_0(x,y)=\frac{y}{1+x}$, $f_1(x,y)=\frac{y}{1+x}+x^3$. It is easily checked that $f_0$ and $f_1$ are ttfun's at $((0,0),dy)$. However, using a theorem of Deligne-Laumon (Theorem \ref{DLthm}), one computes: $\mathrm{dim}(\phi_{f_0}(\mathcal{F})_a)=-(p-1)$, while $\mathrm{dim}(\phi_{f_1}(\mathcal{F})_a)=-2$.

\begin{question}
    What stability do vanishing cycles have in the algebraic context?
\end{question}

One expects that simple normal crossing tame sheaves have similar stability as in the complex analytic context. This turns out to be true (see Theorem \ref{intro_thm_2}). In general, one expects the dependence of vanishing cycles on the ttfun to be only up to a finite jet (as is suggested by a computation similar to the above with $f_1=\frac{y}{1+x}+x^N$ for a big $N$). In the first part of this paper, we show that this is generically true on a smooth surface. This result is inspired by a result of Saito (\cite[\nopp 2.14]{saito_characteristic_2015}). To state it precisely, we first introduce the following notion and some preliminary notations.

\begin{definition}[depth of $\mathcal{F}$]
    Let $\mathcal{F}$ be a sheaf on a smooth variety over a field of characteristic $p>2$. For a smooth point $(x,\xi)$ of $SS\mathcal{F}$, the \underline{depth of $\mathcal{F}$ at $(x,\xi)$} is the smallest $N\geq 2$ such that $\phi_f(\mathcal{F})$ is a local system for every ttfam $(T,U,V,f)$ at $(x,\xi)$ satisfying the following condition: $f_s\equiv f_{s'}$ mod $\mathfrak{m}_x^N$, for all closed points $s, s'$ of $T$. If such an $N$ does not exist, we say the depth is $\infty$.\\

    Here $\phi_f(\mathcal{F})$ is defined to be $\Phi_f(\mathcal{F}_V)|_{x_T\overleftarrow{\times}_{\mathbb{A}^1_T} (\mathbb{A}^1_T-T)}$, the restriction of the vanishing cycle of the pullback $\mathcal{F}_V$ of $\mathcal{F}$ to $V$ with respect to the map $f: V\rightarrow \mathbb{A}^1_T$ to the oriented product topos $x_T\overleftarrow{\times}_{\mathbb{A}^1_T} (\mathbb{A}^1_T-T)$.
\end{definition}

\begin{remark}
    (i) As the transversality of a test function depends on the quadratic terms of the function (in some local coordinate), we exclude $p=2$ throughout the text to avoid unnecessary complications. Also note that when $p\neq 2$, there is enough supply of ttfun's, in the sense that they always exist for every smooth point $(x,\xi)$ of $SS\mathcal{F}$ (\cite[\nopp 4.12]{beilinson_constructible_2016}).\\
    (ii) We refer to \cite{orgogozo_modifications_2006, illusie_around_2017} for details on oriented products and vanishing cycle over general bases, and to Remark \ref{rmkttfam} for further discussion of this definition.
\end{remark}

Let $X$ be a smooth surface over an algebraically closed field $k$ of characteristic $p>2$, and $\mathcal{F}\in D_{ctf}(X)$. Let $D$ be the ramification divisor of $\mathcal{F}$ (assumed to be non-empty), and $U=X-D$ its complement. Let $\overline{U}\rightarrow U$ be the minimal \etale Galois covering trivialising $\mathcal{F}$,\footnote{I.e., the covering corresponding to the quotient $\pi_1(U,\overline{\eta}_U)\twoheadrightarrow G$, where $G$ is the image of $\pi_1(U,\overline{\eta}_U)$ in $\mathrm{Aut}_{\mathbb{Z}/\ell^n}(\mathcal{F}_{\overline{\eta}_U})$.} with Galois group $G$. Let $\overline{X}\rightarrow X$ be the normalisation of $X$ in $\overline{U}$. We use the following notations (see §\ref{subsection on high-jet} for more details): $I_{\sigma, \overline{X}}$ denotes the ideal corresponding to the subscheme of fixed points of $\sigma$ acting on $\overline{X}$; $ep(I_{\sigma, \overline{X}})$ denotes $\mathrm{the\,\, smallest}\,\, a \in \mathbb{N}$ such that $(\sqrt{I_{\sigma, \overline{X}}})^a\subseteq{I_{\sigma, \overline{X}}}$; $i_x$ denotes the intersection number at $x$ of the zero locus of a ttfun at $(x,\xi)$ with $D$ (which is either 1 or 2 for a general $(x,\xi)$).

\begin{theorem}[Theorem \ref{thmstabbeau}, c.f. Theorem \ref{thmstab}]
    Let $X$ be a smooth surface over an algebraically closed field $k$ of characteristic $p>2$, and $\mathcal{F}\in D_{ctf}(X)$. Then, there exists a Zariski open dense $V=X-$\{finitely many closed points\} and a Zariski open dense $S\subseteq SS(\mathcal{F}|_V)$ such that for every closed point $(x,\xi)\in S$, there exists an integer $N\geq2$ such that the depth of $\mathcal{F}$ at $(x,\xi)$ is $\leq N$. Moreover, we have an explicit upper bound: if $\mathcal{F}$ is locally constant in some punctured neighbourhood of $x$, then $N=2$; if $x$ lies in a ramification divisor of $\mathcal{F}$, then $N\leq2^{M-1}\cdot i_x\cdot |G|+(2p+1)^M\cdot\max_{\sigma\neq \mathrm{id}\in G}\{ep(I_{\sigma, \overline{X}})\}\cdot i_x\cdot|G|$. Here $M$ is a uniform bound for the number of blowups needed to resolve the singularities of the curve $f^{-1}(0)\times_X \overline{X}\hookrightarrow\overline{X}$, as $f$ ranges through all ttfun's at $(x,\xi)$. It is part of the claim that this bound exists and can be made explicit.
\end{theorem}

\begin{remark}
    The subsets $V$ and $S$ can be made explicit, see Theorems \ref{thmstab}. The statement is expected to still hold outside of $V$ and $S$ (see Conjecture \ref{intro_conj}), but our method does not apply directly. This requires further investigation.
\end{remark}

Here is an outline of the proof. By dévissage one reduces to the case of a local system !-extended along a divisor. Using a distinguished triangle of Saito (Remark \ref{rmkttfam}.(iv)), one can rephrase the local constancy of $\phi_f(\mathcal{F})$ as the pair $(f_T: V_T\rightarrow T, \mathcal{F}_T)$ being universally locally acyclic (ULA), where $V_T\rightarrow T$ is the family of zero loci of this ttfam, and $\mathcal{F}_T$ is the pullback of $\mathcal{F}$ to $V_T$. Applying the theorem of Deligne-Laumon, we then translate the ULA condition to the Swan conductor on each fibre being constant along the family. This reduces the question to the stability of Swan conductors of restrictions to curves. As the Swan conductor can be computed in terms of intersection numbers and representation-theoretic data (§\ref{subsec_swan}), we further reduce the question to a purely geometric one. We analyse the geometric question by studying how various invariants change under blowups, and finally reduce the question to bounding the blowup number with respect to the variation of the curve. An argument of Bernd Ulrich utilising Dedekind codifferents then shows that this number is indeed bounded.\\

We conjecture that the theorem holds in greater generality. The positive answer to this conjecture will enable one to define microstalks (potentially even microlocalisation) in positive characteristic, on a high jet bundle.
\begin{conjecture}[Conjectures \ref{finitedepth}, \ref{finitedepthrep}]\label{intro_conj}
     Let $X$ be a smooth variety over an algebraically closed field $k$ of characteristic $p\neq 2$. Then every $\mathcal{F}\in D(X)$ has finite depth at all smooth points of $SS\mathcal{F}$.
\end{conjecture}

The second part of this paper studies the class of sheaves whose vanishing cycles have the strongest stability, as well as certain functorialities of the depth. We call a sheaf $\mu c$ if it has depth 2 at all smooth points in its $SS$, and $\mu c^s$ if it is $\mu c$ after all smooth pullbacks (Definition \ref{defmucsheaves}). The stability of vanishing cycles for $\mu c$ and $\mu c^s$ sheaves is similar to that in the complex analytic context. More precisely, among other things, we show:

\begin{theorem}[Proposition \ref{tamearemuc}, Lemma \ref{phiindepofttfun}.(ii), Corollary \ref{Radoncompat}]\label{intro_thm_2}~\\
    (i) Let $X$ be a smooth variety over an algebraically closed field of characteristic $p>2$, $D\hookrightarrow X$ a simple normal crossing divisor (allowed to be empty), and $j: U\hookrightarrow X$ its complement. Then, for every local system $\mathcal{L}$ on $U$, $j_!\mathcal{L}$ is $\mu c^s$.\\
    (ii) Let $X$ be a smooth variety over an algebraically closed field of characteristic $p>2$,  $\mathcal{F}$ a $\mu c$ sheaf on $X$, and $(x,\xi)$ a smooth point in $SS\mathcal{F}$. Then, for every two ttfun's $f$ and $g$ at $(x, \xi)$, there exists (noncanonically) an isomorphism $\phi_f(\mathcal{F})_x\cong \phi_g(\mathcal{F})_x$ as objects in $D^b_c(\mathbb{Z}/\ell^n)$. We call this isomorphism class the microstalk of $\mathcal{F}$ at $(x,\xi)$.\\
    (iii) Being $\mu c^s$ is preserved under the Radon transform. Moreover, their microstalks are invariant under the Radon transform.
\end{theorem}


We outline the proofs. For (i), as for the previous theorem, we first rephrase the question as showing $(f_T: V_T\rightarrow T, \mathcal{F}_T)$ being ULA. We then need to understand the singularities of the intersections of the zero loci of ttfun's and the simple normal crossing divisors. We give explicit resolutions of such singularities (§\ref{subsec_resolution}). Then, $\mathcal{F}_T$ can be written as the pushforward of its pullback via the resolution map $\pi$. The map $f_T\pi$ will be transversal to $SS(\pi^*\mathcal{F}_T)$, from which the ULA statement follows easily. Statement (ii) follows from the definition of $\mu c$ sheaves plus the lemma that every two ttfun's can be connected by some ttfam. For (iii), the argument is similar to the complex analytic case and essentially reduces to a detailed understanding of the geometry of the Radon transform. Actually, one can show the stability of $\mu c$ and $\mu c^s$ sheaves for all proper pushforwards which share similar geometric properties as (the pushforward part of) the Radon transform (Proposition \ref{misclemmaspepush}).\\ 

In the appendix, we list some analogies and contrasts among several sheaf theories from the microlocal point of view.


\subsection*{Conventions}\label{sec_conventions}
     All derived categories are in the triangulated sense. All functors are derived. A “sheaf” means an object of $D(X)$ (see below). A “local system” means an object of $D(X)$ whose cohomology sheaves are locally constant constructible. \\
     
     In the complex analytic context, $D(X)$ denotes $D^b_{\mathbb{C}-c}(X, \mathbb{C})$ in the sense of \cite[\nopp 8.5]{kashiwara_sheaves_1990}. We fix a generator of $\pi_1(\mathbb{C}^{\times},1)$ and identify it with $\mathbb{Z}$.\\
     
     In the algebraic context, we work with varieties (finite type reduced separated schemes over $k$) over an algebraically closed field $k$ of characteristic $p\geq 0$, $D(X)$ denotes $D^b_c(X,\mathbb{Z}/\ell^n)$ for a fixed prime $\ell\neq p$, and $D_{ctf}(X)\subseteq D(X)$ denotes the full subcategory of objects of finite tor-dimension. A “geometric point” means a map from the spectrum of a separably closed field. The fundamental group $\pi_1(\mathbb{A}^1_{k,(0)}-\{0\},\overline{\eta})$ is denoted by $G_\eta$, where $\mathbb{A}^1_{k,(0)}$ is the strict henselisation of $\mathbb{A}^1_k$ at the origin, $\overline{\eta}$ is a fixed geometric point over its generic point.\\
     
     The triangulated category of bounded complexes of $\mathbb{C}$-vector spaces (resp. $\mathbb{C}[\mathbb{Z}]$-modules) with finite dimensional cohomologies is denoted by $D^b_c(\mathbb{C})$ (resp. $D^b_c(\mathbb{C}[\mathbb{Z}])$). Similarly for $D^b_c(\mathbb{Z}/\ell^n)$, $D^b_c(\mathbb{Z}/\ell^n[G_{\eta}])$. For $M\in D^b_{ctf}(\mathbb{Z}/\ell^n[G_{\eta}])$, $\mathrm{sw}(M)$ (resp. $\mathrm{dim}(M)$) denotes the alternating sum of the Swan conductor (resp. the dimension) of its cohomologies, after applying $-\otimes_{\mathbb{Z}/\ell^n}\mathbb{F}_\ell$ (derived tensor).\\
     
     For $f: X\rightarrow Y$ a map of complex analytic manifolds (resp. smooth schemes over $k$), we have the correspondence $T^*X\leftarrow X\times_Y T^*Y\rightarrow T^*Y$. We use $df$ to denote the first map, unless otherwise stated in specific contexts. When $Y=\mathbb{A}^1$, we use $\Gamma_{df}$ to denote the image of $X\times_{\mathbb{A}^1}\{1\}$($\subseteq X\times_{\mathbb{A}^1} T^*\mathbb{A}^1\simeq X\times_{\mathbb{A}^1}\mathbb{A}^1$) in $T^*X$ (as a reduced closed subspace). For a closed conical subset $C\subseteq T^*X$, we refer to \cite{beilinson_constructible_2016} for the meaning of $f_{\circ}C$, $f^{\circ}C$, $C$-transversality, and related terminologies. When speaking of points in $C$, we always mean closed points.\\

     For two subvarieties $C, D$ intersecting at finitely many points in some ambient variety, we use $C\cdot D$ to denote their intersection, either as a subscheme or the intersection number, depending on the context.

\subsection*{Acknowledgement}
I would like to express my sincere gratitude to my PhD advisor David Nadler for his guidance, generosity, and encouragements.\\
I thank Owen Barrett, Sasha Beilinson, Mark Macerato, Martin Olsson, Takeshi Saito, Jeremy Taylor and Bernd Ulrich for valuable discussions, especially to Owen Barrett for Lemma \ref{owenretraction}, and to Bernd Ulrich for his crucial help in proving Proposition \ref{prop_boundonBUstages}. The latter allows to remove a genericity assumption in the stability theorem in a previous version of this article. I also thank Hélène Esnault for correcting a statement in the Appendix, Extension properties.\\
Finally, I heartily thank the referee for their careful reading, corrections, and numerous very helpful suggestions for improving the article. 

\section{Preliminary discussion on sheaf theory}\label{section-review}
In this section, we review some microlocal-sheaf-theoretic constructions in both complex analytic and algebraic contexts, and compare them. Except for the definitions of the ttfun and the Radon setup, §\ref{reviewcomplex} is logically independent of the rest of the paper, but serves as a motivation.
\subsection{Complex analytic context}\label{reviewcomplex}
The reference for this subsection is \cite{kashiwara_sheaves_1990}. Let $X$ be a complex analytic manifold, and $D(X)$ be the triangulated category of bounded $\mathbb{C}$-constructible complexes of sheaves of $\mathbb{C}$-vector spaces. The notion of the singular support (or the microsupport) $SS\mathcal{F}$ is defined for $\mathcal{F}\in D(X)$. It is a middle dimensional $\mathbb{C}^{\times}$-conical closed \emph{Lagrangian} subset in $T^*X$ which records the codirections in which $\mathcal{F}$ is not locally constant. More precisely, it equals the closure of all $(x,\xi)\in T^*X$ such that there exists some complex analytic function $f$ on some open neighbourhood of $x$ such that the stalk of the vanishing cycle $\phi_f(\mathcal{F})_x$ is nonzero. $SS\mathcal{F}$ is the 0-th order invariant (the locus) of the “singularities” of $\mathcal{F}$. Clearly, the vanishing cycle, viewed as an object in $D^b_c(\mathbb{C}[\mathbb{Z}])$ (bounded complexes of $\mathbb{C}[\mathbb{Z}]$-modules with finite dimensional cohomologies), is a much finer measurement of the sheaf. However, it depends on the choice of the test function $f$. It is a crucial fact that, when restricted to transverse test functions, $\phi_f(\mathcal{F})_x$ is essentially independent of $f$, in the precise sense below. We refer to Definitions \ref{ttfun}, \ref{Cttfamdef} for the definitions of ttfun and ttfam.

\begin{theorem}\label{phistability/C}
Let $X$ be a complex analytic manifold, $\mathcal{F}\in D(X)$, $(x, \xi)$ a smooth point of $ SS\mathcal{F}$. Then:\\
(i) For every two ttfun's $f$ and $g$ of $\mathcal{F}$ at $(x, \xi)$, there exists (noncanonically) an isomorphism $\phi_f(\mathcal{F})_x\cong \phi_g(\mathcal{F})_x$ in $D^b_c(\mathbb{C}[\mathbb{Z}])$.\\
(ii) For every ttfam $(T,U,V,f)$ of $\mathcal{F}$ at $(x, \xi)$, $\phi_{pf}(\mathcal{F}_V)$ is a local system on $x_T$, with the stalk at $x\times s$ canonically isomorphic to $\phi_{f_s}(\mathcal{F})_x$, for all $s\in T$ $(=0\times T\subseteq \mathbb{A}^1\times T)$. Here $p$ is the projection $\mathbb{A}^1\times T\rightarrow\mathbb{A}^1$, $\mathcal{F}_V$ is the pullback of $\mathcal{F}$ to $V$.
\end{theorem}

In the following discussion, we will need a variant of ttfam: in Definition \ref{Cttfamdef}, as $s$ varies, instead of requiring $f_s$ to be ttfun's at a fixed $\nu_0=(x,\xi)$, we allow $f_s$ to be ttfun's at $\nu(s)=(x(s),\xi(s))$ for varying smooth points $\nu(s)$ on $SS\mathcal{F}$, and require $\nu(s_0)=\nu_0$ for some $s_0$. We call such families \underline{weak transverse test families (wttfam)} at $(x,\xi)$.\\

The real analytic counterpart of Theorem \ref{phistability/C} is contained in the statement and proof of \cite[\nopp 7.2.4]{kashiwara_microlocal_1985}. The complex case can be easily deduced from it. We include a proof for completeness.\\

For a real analytic manifold $X$, $D(X)$ denotes the triangulated category of bounded $\mathbb{R}$-constructible sheaves of $\mathbb{C}$-vector spaces. For $f$ a real analytic function and $\mathcal{F}\in D(X)$, the vanishing cycle is defined as $\phi_f(\mathcal{F})=R\Gamma_{\{f\geq0\}}(\mathcal{F})|_H$, where $H=\{f=0\}$. If $H$ is smooth, it is also equal to $d_f^*\mu_Y(\mathcal{F})$,\footnote{We use the notation “$^*$” instead of “$^{-1}$” (as in \cite{kashiwara_sheaves_1990}) for the sheaf pullback.} where $\mu_Y$ is the microlocalisation along $Y$, $d_f$ is the map $Y\rightarrow T^*_YX$, $y\mapsto (y,df)$. The notions of ttfun, ttfam and wttfam have obvious analogues in the real analytic context: namely, one replaces the word “complex” (resp. “$\mathbb{C}$”) by “real” (resp. “$\mathbb{R}$”) in Definition \ref{ttfun} and Definition \ref{Cttfamdef}.

\begin{proposition}[{\cite[\nopp 7.2.4]{kashiwara_microlocal_1985}}]\label{KS85,7.2.4}
    Let $X$ be a real analytic manifold, $\mathcal{F}\in D(X)$, and $(x, \xi)$ a smooth point of $SS\mathcal{F}$. Then for every wttfam $(T,U,V,f)$ of $\mathcal{F}$ at $(x, \xi)$, $\phi_{pf}(\mathcal{F}_V)$ is a local system on $x_T$, with the stalk at $x\times s$ canonically isomorphic to $\phi_{f_s}(\mathcal{F})_x$, for all $s\in T$.
\end{proposition}

\begin{proof}
    The proof in \cite[\nopp 7.2.4]{kashiwara_microlocal_1985} works for $\xi\neq 0$. We sketch the argument. Let $H=\{pf=0\}, H_s=\{f_s=0\}$. We have $\phi_{pf}(\mathcal{F}_V)=d_{pf}^*\mu_H(\mathcal{F}_V), \phi_{f_s}(\mathcal{F})=d_{f_s}^*\mu_{H_s}(\mathcal{F})$. Let $W=SS\mathcal{F}_V\cap T^*_HV=\mathbb{R}_{> 0}\{(s,x,d(f_s))\}_{s\in T}, W_s=SS\mathcal{F}\cap T^*_{H_s}V_s=\mathbb{R}_{> 0}\{(s,x,d(f_s))\}$. By the estimate of $SS$ of microlocalisations (\cite[\nopp 5.2.1.(ii)]{kashiwara_microlocal_1985}), one checks that $SS(\mu_H(\mathcal{F}_V))\subseteq T^*_WT^*_HX$. This implies $\mu_H(\mathcal{F}_V)|_W$ is locally constant. The first statement follows. By functoriality of the microlocalisation under noncharacteristic pullbacks (\cite[\nopp 5.4.2]{kashiwara_microlocal_1985}), we get $\mu_H(\mathcal{F}_V)|_{W_s}\cong\mu_{H_s}(\mathcal{F})$. The second statement follows.\\
    
    For $\xi=0$. Consider the embedding $i: X=X\times \{0\}\hookrightarrow X\times \mathbb{R}$. Let $z$ be the standard coordinate on $\mathbb{R}$. One checks that the family of functions $\{z-f_s\}_{s\in T}$ gives a wttfam for $i_*\mathcal{F}$ at $(x,\xi')$, where $\xi'$ is any nonzero conormal vector at $x$ of $X$ in $X\times \mathbb{R}$. Then the previous case applies, and the compatibility of vanishing cycles and proper pushforwards implies this case.
\end{proof}

For a complex analytic manifold $Y$, denote by $Y^{\mathbb{R}}$ the underlying real analytic manifold, there is a canonical identification $(T^*Y)^{\mathbb{R}}=T^*Y^{\mathbb{R}}$ (see, e.g., \cite[\nopp 11.1]{kashiwara_sheaves_1990}). For a complex analytic function $h$, $(\Gamma_{dh})^{\mathbb{R}}=\Gamma_{\mathrm{Re}(h)}$ under this identification.  In particular, if $h$ is a ttfun for some $\mathcal{F}$ then so is $\mathrm{Re}(h)$. Furthermore, by \cite[\nopp 13]{kashiwara_sheaves_1990}, for a general $h$ we have a canonical isomorphism $\phi_{h}\cong \phi_{\mathrm{Re}(h)}|_H$ in $D(H)$, where $H=\{\mathrm{Re}(h)=0\}$.

\begin{proof}[Proof of Theorem \ref{phistability/C}]
     (ii) This is immediate from the above paragraph and Proposition \ref{KS85,7.2.4}: a ttfam on $X$ induces a ttfam on $X^{\mathbb{R}}$ whose vanishing cycle is a local system with stalks isomorphic to the vanishing cycles on the slices. Transfer back to complex vanishing cycles, we get the result.\\
     
     (i) This follows from the following observation: given any ttfam $(T,U,V,f)$ on $X$, consider the family $((T\times \mathbb{C}^{\times})^{\mathbb{R}},U^{\mathbb{R}},(V\times\mathbb{C}^{\times})^{\mathbb{R}},g)$ on $X^{\mathbb{R}}$, which on each slice $(s,\lambda)\in (T\times \mathbb{C}^{\times})^{\mathbb{R}}$ is given by $g_{(s,\lambda)}=\mathrm{Re}(\lambda f_s)$. One checks this is a wttfam. By Proposition \ref{KS85,7.2.4}, $\phi_{pg}(\mathcal{F}_{(V\times\mathbb{C}^{\times})^{\mathbb{R}}})$ is a local system on $(T\times \mathbb{C}^{\times})^{\mathbb{R}}$ with stalks at $(s,\lambda)$ isomorphic to $\phi_{\mathrm{Re}(\lambda f_s)}(\mathcal{F})_x$. Moreover, $\phi_{\mathrm{Re}(\lambda f_s)}(\mathcal{F})_x$ viewed as a local system with respect to $\lambda$ (i.e. $\phi_{pg}(\mathcal{F}_{(V\times\mathbb{C}^{\times})^{\mathbb{R}}})|_{s\times \mathbb{C}^{\times}}$) is exactly $\phi_{f_s}(\mathcal{F})_x$ viewed as a local system on $\mathbb{C}^{\times}$. This implies $\phi_{f_s}(\mathcal{F})_x\cong\phi_{f_s'}(\mathcal{F})_x$ (noncanonically) for any $s,s'\in T$.\\
     
     So, to show (i), it suffices to show that any two ttfun's can be connected by a ttfam. This is a simple exercise: fix a coordinate, expand a ttfun in power series, cut off degree $\geq 3$ terms with a ttfam\footnote{Note that being a ttfun only depends on degree $\leq 2$ terms.}, then observe that the space of all quadratic terms which makes the function a ttfun is a connected complex analytic manifold. (See proof of Lemma \ref{phiindepofttfun} (i) for a detailed argument in the algebraic context.)
\end{proof}

As mentioned in the introduction, Theorem \ref{phistability/C} is a fundamental fact underlying many microlocal-sheaf-theoretic constructions. In particular, to every smooth point $(x, \xi)$ in $SS\mathcal{F}$, this allows us to define the \underline{microstalk ($\mu$\textit{stalk})} of $\mathcal{F}$ at $(x,\xi)$: take $\phi_f(\mathcal{F})_x$ for any ttfun $f$ at $(x,\xi)$. It is an object in  $D^b_c(\mathbb{C}[\mathbb{Z}])$, independent of $f$ in the sense above.\\

Another (related) fundamental feature of real and complex analytic microlocal sheaf theory is its invariance under contact transformations, of which the Radon transform is the prototypical example. We will not discuss the full invariance, but focus on one aspect of it: how microstalks change under the Radon transform.
\begin{radonsetup}[for both complex analytic and algebraic contexts]\label{radonsetup} 
\[\begin{tikzcd}
	& Q \\
	{\mathbb{P}} && {\mathbb{P}^{\vee}}
	\arrow["\p"', from=1-2, to=2-1]
	\arrow["\q", from=1-2, to=2-3]
\end{tikzcd}\]
Here, $n\geq 1$, $\mathbb{P}$ is the abbreviation for $\mathbb{P}^n$ (over the base field), $\mathbb{P}^{\vee}$ is its dual, $Q$ is the universal incidence variety. Let $\mathcal{F}\in D(\mathbb{P})$. Its Radon transform is defined as $R\mathcal{F}=\q_!\p^*\mathcal{F}[n-1]$. Denote by $P(...)$ the projectivisation of $(...)$ after removing the zero section. We have the following facts (see \cite[\nopp 1.6, 3.3]{beilinson_constructible_2016} \footnote{Strictly speaking, \cite{beilinson_constructible_2016} is only in the algebraic context. But the facts also hold in the analytic context: this is clear for (i) and (ii), and the same proof as in \cite[\nopp 3.3]{beilinson_constructible_2016} gives (iii).}):\\
(i) $P(T^*\mathbb{P})\cong Q \cong P(T^*\mathbb{P}^{\vee})$;\\
(ii) Let $z$ be a point in $Q$, $x=\p(z), a=\q(z)$, and let $\xi, \alpha$ be nonzero covectors at $x, a$ which are conormal to the hyperplanes represented by $a, x$,  respectively. Then $z$ is the codirection represented by $\xi, \alpha$ under the identifications in (i). Furthermore, $T_z^*Q$ equals the pushout of $T_x^*\mathbb{P}$ and $T_a^*\mathbb{P}^{\vee}$ along $\langle\xi\rangle$ and $\langle\alpha\rangle$ (via $d\p_x$ and $d\q_a$). We say $(x, \xi)$ and $(a, \alpha)$ correspond to each other;\\
(iii) $SS^+(R\mathcal{F})=\q_{\circ}SS^+(\p^*\mathcal{F})=\q_{\circ}\p^{\circ}SS^+\mathcal{F}$, where $^+$ means adding the zero section.  $PSS\mathcal{F}=PSSR\mathcal{F}$ as subvarieties of $Q$.
\end{radonsetup}

\begin{proposition}\label{radonstability/C}
    With the complex analytic Radon setup \ref{radonsetup}, let $\nu$ be a smooth point in $PSS\mathcal{F}=PSSR\mathcal{F}$. Then, $\mu stalk(R\mathcal{F})_{\nu} \cong \mu stalk(\mathcal{F})_{\nu}$ if $n$ is odd, and $\mu stalk(R\mathcal{F})_{\nu} \cong \mu stalk(\mathcal{F})_{\nu} \otimes \mathcal{K}_2$ if $n$ is even. Here $\mathcal{K}_2 \in D^b_c(\mathbb{C}[\mathbb{Z}])$ is the vector space $\mathbb{C}$ concentrated in degree $0$, with $1\in\mathbb{Z}$ acting by multiplication by $-1$.
\end{proposition}
In particular, as vector spaces (i.e. as objects in $D^b_c(\mathbb{C})$), microstalks are invariant for all $n$. We will later prove a similar result in the algebraic context (Proposition \ref{phistabradon}). The same argument proves Proposition \ref{radonstability/C}. More precisely, in the argument, one replaces Proposition \ref{dimtotstab} by Theorem \ref{phistability/C}, the Thom-Sebastiani Theorem by its complex analytic counterpart (see, e.g., \cite[Theorem 1.0.1]{schurmann_topology_2003}), and computes $\phi_{y_1^2+...+y_{n-1}^2}(\underline{\mathbb{C}})_0$ inductively using Thom-Sebastiani. We omit the details. This result is not used in the sequel.

\subsection{Algebraic context}\label{sectiononfailure}
We now consider the algebraic context. Let $X$ be a smooth variety over a field $k$ algebraically closed of characteristic $p\geq0$, and $D(X)$ the triangulated category of bounded constructible complexes of étale sheaves of $\mathbb{Z}/\ell^n$-modules. The notion of the singular support $SS\mathcal{F}$ is defined for $\mathcal{F}\in D(X)$ (\cite{beilinson_constructible_2016}). It is a middle dimensional conical closed subset in $T^*X$. As in the analytic case, it records the non-locally-acyclic codirections of $\mathcal{F}$, and has a similar description in terms of test functions and vanishing cycles. The notion of the ttfun (Definition \ref{ttfun}) has the following obvious analogue in this context (see Definition \ref{ttfamdef} for the notion of the ttfam, which is not used in this subsection):
\begin{definition}[transverse test function (algebraic)]\label{ttfun-alg}
    A \underline{transverse test function (ttfun)} of $\mathcal{F}$ at a \emph{smooth} point $(x, \xi)$ of $SS\mathcal{F}$ is a regular function $f$ defined on an open neighbourhood U of $x$ such that the following are satisfied:\\
(i) $f(x)=0$;\\
(ii) the graph $\Gamma_{df}$ of the differential of $f$ intersects $SS(
\mathcal{F})$ only at $(x, \xi)$, and the intersection is transverse.
\end{definition}

In the positive characteristic world, in contrast to the above, singular supports need not be Lagrangian\footnote{Actually, Deligne (\cite{deligne_letter_2015}) showed that on a smooth surface $X$, any middle dimensional conical closed subset in $T^*X$ can be generically realised as a component of some $SS\mathcal{F}$.}. We will later record more new phenomena (§\ref{subsection-example}). In this section, we discuss the failure of the analogue of Theorem \ref{phistability/C}.\\

We will use the following result of Deligne and Laumon to compute the dimensions of vanishing cycles (\cite[\nopp 2.1, 5.1]{laumon_semi-continuite_1981}, see \cite[\nopp 2.12]{saito_characteristic_2017} for a more general version):

\begin{theorem}[Deligne-Laumon]\label{DLthm}
    Let $S$ be a Noetherian excellent scheme, $f: X \rightarrow S$ a separated smooth morphism of relative dimension 1, and $Z$ a closed subscheme of $X$ finite flat over $S$ with a single point in each fibre. Let $\mathcal{F}\in D(X)$ be the !-extension of a tor-finite locally constant sheaf concentrated in degree 0 on $U=X-Z$. Define the $\mathbb{N}$-valued function $a_s$ on the points of $S$: \begin{equation}\label{DLeqn1}
    a_s:= \mathrm{dimtot}((\mathcal{F}|_{X_{\overline{s}}})_{\overline{x}_{\overline{z}}})
    \end{equation} 
    where $\overline{s}$ is a geometric point over $s$ with residue field an algebraic closure of the residue field of $s$, $\overline{z}$ is a geometric point of $Z$ above $\overline{s}$, $\overline{x}_{\overline{z}}$ is a geometric point over the generic point of 
    the strict henselisation of
    $X_{\overline{s}}$ at $\overline{z}$, $\mathrm{dimtot}$ means $\mathrm{sw}+\mathrm{dim}$ (see Conventions).\\\\
    Then:\\
    (i) $a_s$ is constructible, and $a_s\leq a_{\eta}$ if $\eta$ specialises to $s$.\\
    (ii) ($f$,$\mathcal{F}$) is universally locally acyclic (ULA) if and only if $a_s$ is locally constant.\\
    (iii) If $S$ is an excellent strict henselian trait, denote its closed and generic points by $s,\eta$ respectively, then
    \begin{equation}\label{DLeqn2}
        a_s-a_{\eta}=\mathrm{dim}(\phi_f(\mathcal{F}_{\overline{z}})).
    \end{equation}
 \end{theorem}

We now give examples showing the failure of the analogue of Theorem \ref{phistability/C}.
 
\begin{example}\label{ex1}
($p>2$, $\mathbb{Z}/\ell$-coefficient) Let $X=\mathbb{A}^2=\mathrm{Spec}( k[x,y])$. Let $\mathcal{F}$ be the Artin-Schreier sheaf determined by the equation $t^p-t=y/x^p$ and a nontrivial character $\mathbb{Z}/p\rightarrow(\mathbb{Z}/\ell)^{\times}$,\footnote{This equation determines a finite \etale Galois covering of $U=X-\{x=0\}$, with Galois group $\mathbb{Z}/p$, corresponding to a surjection $\pi_1(U,\overline{\eta}_U)\twoheadrightarrow\mathbb{Z}/p$. Composing with the character gives a representation of $\pi_1(U,\overline{\eta}_U)$, which is the same thing as a local system on $U$. (We assume such characters exist, that is, $p|(\ell-1)$).} !-extended along $D=\{x=0\}$. One can show $SS\mathcal{F}=T^*_XX \cup \langle dy\rangle_D$, where $\langle dy\rangle_D$ denotes the subspace of $D\times_X T^*X$ consisting of covectors proportional to $dy$ (see \cite[\nopp 3.6]{saito_wild_2017} or Example \ref{cancelss}). Consider the following family of ttfun's at $\nu=((0,0),dy)$: $f_s(x,y):=\frac{y}{1+x}+sx^N$, where $N$ is some integer $\geq$ 3, $s\in k$. It is simple to check that for each fixed $s$, $f_s$ (restricted to some Zariski neighbourhood of $(0,0)$) is indeed a ttfun.
\end{example}

For a fixed $s$, apply Deligne-Laumon to $f_s: U \rightarrow \mathbb{A}^1$, where $U$ is some Zariski open neighbourhood of $(0,0)$ on which $f_s$ is defined. Let $\rho$ be the standard coordinate on $\mathbb{A}^1$. The fibre $f_s^{-1}(\rho)$ is locally isomorphic to $\mathbb{A}^1$ with $x$ as a coordinate. $\mathcal{F}|_{f_s^{-1}(\rho)}$ is the Artin-Schreier sheaf determined by $t^p-t=(\rho-sx^N)(1+x)/x^p$ (!-extended at $x=0$). In formula (\ref{DLeqn1}), $\mathrm{dim}(\mathcal{F}_{\overline{z}})=0, \mathrm{dimtot}((\mathcal{F}|_{X_{\overline{s}}})_{\overline{x}_{\overline{z}}})=1+\mathrm{sw}((\mathcal{F}|_{X_{\overline{s}}})_{\overline{x}_{\overline{z}}})$. The Swan conductors are easily computed:\\

\begin{minipage}[c]{0.5\textwidth}
\centering
\begin{tabular}{|l|l|l|}
\hline
sw$(\rho, s)$ & $s=0$ & $s$ generic \\ 
\hline
$\rho=0$                                                    & 0   & $p-N$       \\ 
\hline
$\rho$ generic                                              & $p-1$ & $p-1$       \\ 
\hline
\end{tabular}
\captionof*{table}{$3\leq N < p$}
\end{minipage}
\begin{minipage}[c]{0.5\textwidth}
\centering
\begin{tabular}{|l|l|l|}
\hline
sw$(\rho, s)$ & $s=0$ & $s$ generic \\ 
\hline
$\rho=0$                                                    & 0   & 0       \\ 
\hline
$\rho$ generic                                              & $p-1$ & $p-1$       \\ 
\hline
\end{tabular}
\captionof*{table}{$N \geq p$}
\end{minipage}\\\\
 
By formula (\ref{DLeqn2}), the dimensions of $\phi_{f_s}(\mathcal{F})$ are as follows:

\begin{table}[H]
\begin{center}
\begin{tabular}{|l|l|l|}
\hline
dim($\phi_{f_s}(\mathcal{F})$) & $s=0$ & $s$ generic \\ 
\hline
$3\leq N < p$                                                    & $-(p-1)$   & $-(N-1)$       \\ 
\hline
$N \geq p$                                              & $-(p-1)$ & $-(p-1)$       \\ 
\hline
\end{tabular}
\end{center}
\end{table}
 
We see that if $p>3$ and $3\leq N < p$, then dim($\phi_{f_s}(\mathcal{F})$) depends on the parameter $s$. So the analogue of Theorem \ref{phistability/C} is false. Nevertheless, if $N \geq p$, then dim($\phi_{f_s}(\mathcal{F})$) does not depend on $s$ (for $s$ in a small neighbourhood of $0\in\mathbb{A}^1$). This is a first indication that vanishing cycles depend on the ttfun only up to a finite jet. We will come back to this in §\ref{subsection on high-jet}.\\

In this example, $SS\mathcal{F}$ is not Lagrangian. Does the analogue of Theorem \ref{phistability/C} hold if restricted to sheaves whose $SS$'s are Lagrangian? The answer is no, as the next example shows:

\begin{example}\label{ex2}
    Same setup and notations as above, but consider the Artin-Schreier sheaf determined by $t^p-t=y/x^{p-1}$. One can show $SS\mathcal{F}=T^*_XX \cup T^*_{(0,0)}X \cup \langle dx\rangle_D$ (see Example \ref{cancelss}). Consider the same $\nu$ and the same family of ttfun's as above.
\end{example}
 
The computation is similar, the results are as follows:\\

\begin{minipage}[c]{0.5\textwidth}
\centering
\begin{tabular}{|l|l|l|}
\hline
sw$(\rho, s)$ & $s=0$ & $s$ generic \\ 
\hline
$\rho=0$                                                    & 0   & $p-N-1$       \\ 
\hline
$\rho$ generic                                              & $p-1$ & $p-1$       \\ 
\hline
\end{tabular}
\captionof*{table}{$3\leq N < p-1$}
\end{minipage}
\begin{minipage}[c]{0.5\textwidth}
\centering
\begin{tabular}{|l|l|l|}
\hline
sw$(\rho, s)$ & $s=0$ & $s$ generic \\ 
\hline
$\rho=0$                                                    & 0   & 0       \\ 
\hline
$\rho$ generic                                              & $p-1$ & $p-1$       \\ 
\hline
\end{tabular}
\captionof*{table}{$N \geq p-1$}
\end{minipage}

\begin{table}[H]
\begin{center}
\begin{tabular}{|l|l|l|}
\hline
dim($\phi_{f_s}(\mathcal{F})$) & $s=0$ & $s$ generic \\ 
\hline
$3\leq N < p-1$                                                    & $-(p-1)$   & $-N$       \\ 
\hline
$N \geq p-1$                                              & $-(p-1)$ & $-(p-1)$       \\ 
\hline
\end{tabular}
\end{center}
\end{table}

We remark that the simplicity of the fundamental group (in particular, that it splits locally, and that all ramifications are tame) is one fundamental reason why in the analytic context vanishing cycles have strong stability, strong enough that they “live” on the cotangent bundle, leading to fundamental constructions in microlocal sheaf theory. In the positive characteristic algebraic context, due to the complexity of the \etale fundamental group (or wild ramifications), the (micro)local data of a sheaf is huge. This is analogous to the distinction between regular holonomic $D$-modules and general holonomic $D$-modules. In the appendix, we list some more analogies and distinctions.

\section{Preliminary discussion on algebraic geometry}
This section discusses some results in algebraic geometry that will be used later. They are crucial in the following, but can be skipped at first reading. In this section, $k$ is algebraically closed of characteristic $p>0$, and, for a variety $X$, a sheaf means an object of $D(X)$ (see Conventions). 

\subsection{The Swan conductor}\label{subsec_swan}
References for this paragraph are \cite[\nopp 1.1]{laumon_semi-continuite_1981} and \cite[\nopp VI.4]{serre_local_1979}. Let $C$ be a strict henselian trait, and $\mathcal{L}$ a tor-finite sheaf at its generic point $\eta$ concentrated in degree 0, given by a Galois representation $G_{\eta}\twoheadrightarrow G\hookrightarrow \mathrm{Aut}_{\mathbb{Z}/\ell^n}(\mathcal{L}_{\overline{\eta}})$. Let $C'\rightarrow C$ be the normalisation of $C$ in the Galois covering of $\eta$ corresponding to $G$, note $C'$ is a trait. The Swan conductor $\mathrm{sw}(\mathcal{L})$ of $\mathcal{L}$ is a non-negative integer measuring the wild ramification of $\mathcal{L}$. It can be computed as follows: consider the filtration $G=G_0\supseteq G_1\supseteq...$ induced by $i_G: G\rightarrow \mathbb{N}\cup\{\infty\}, \sigma\mapsto v'(\sigma(\pi')-\pi')$ if $\sigma\neq \mathrm{id}$; $\infty$ if $\sigma=\mathrm{id}$, where $\pi'$ is any uniformiser of $C'$, $v'$ is the discrete valuation on $C'$, and $G_i=\{\sigma\in G|\,\, i_G(\sigma)\geq i+1\}$. Then $$\mathrm{sw}(\mathcal{L})=\sum_{i\geq 1}\frac{\mathrm{dim}(\mathcal{F}_{\overline{\eta}}/\mathcal{F}_{\overline{\eta}}^{G_i})}{[G:G_i]}$$

The following geometric interpretation of $i_G$ will be important for us. Consider the $G$-action on $C'$. For $\sigma\neq \mathrm{id}$, we have $i_G(\sigma)=(\Gamma_{\sigma}\cdot \Delta_{C'})$, where the latter is the intersection number of the graph of $\sigma$ and the diagonal. Note that, if $I_{\sigma,C'}$ is the ideal on $C'$ corresponding to this intersection, then $(\Gamma_{\sigma}\cdot \Delta_{C'})=\lambda_{\mathcal{O}_{C'}}(\mathcal{O}_{C'}/I_{\sigma,C'})$, where $\lambda$ denotes the length as a module.

\subsection{Blowup bounds}\label{subsec_blowup_bound}
The materials in this subsection will be used in the proof of Theorem \ref{thmstab}.

\begin{terminology}\label{BUstages}
    Let $X=X_0\xleftarrow{a_1} X_1\xleftarrow{a_2}X_2\xleftarrow{a_3}\cdots\xleftarrow{a_n}X_n$ be a blowup sequence of length $n$ of a smooth variety $X$, such that each $a_i$ is a blowup at finitely many closed points of $X_{i-1}$. If, for every $2\leq i\leq n$, $a_i$ is the blowup at a single point lying in the exceptional divisor created by $a_{i-1}$, then we say this blowup sequence is monotone and has $n$ \underline{blowup stages}. In general, we say it has $r$ blowup stages if, among all renumberings of the blowups such that each blowup is at a single point, the longest monotone subsequence of blowups (necessarily starting from $X$) has length $r$.
\end{terminology}

For example, the following blowup sequence of length 4 has 3 blowup stages.

\begin{figure}[H]
\centering
\includegraphics[width=0.75\textwidth]{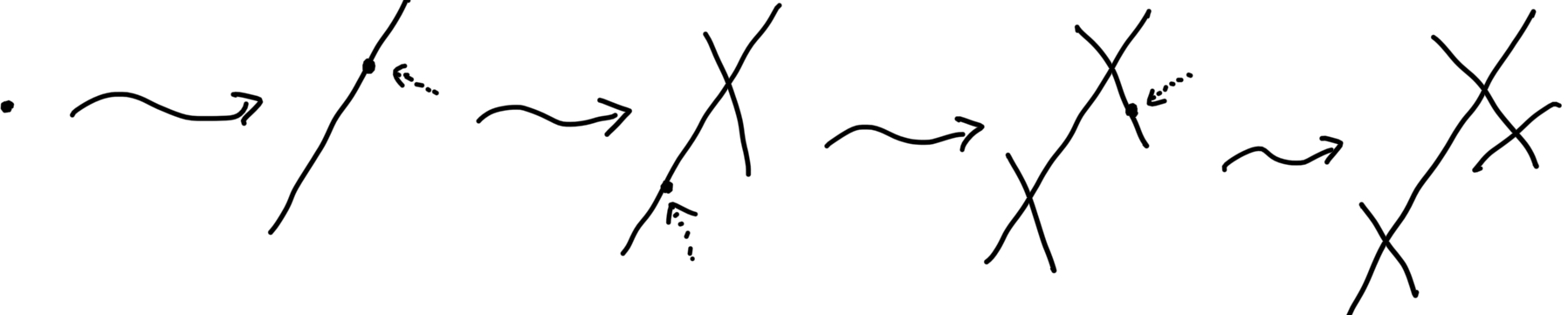}
\caption{A 3-stages blowup sequence on a surface\label{fig:BUstages}}
\end{figure}

\begin{definition}[$ep(I)$]
    Let A be a Noetherian ring, and I an ideal. Define $ep(I)$ (ep stands for épaisseur) to be the smallest $r \in \mathbb{N}$ such that $(\sqrt{I})^r\subseteq{I}$. Note $ep(I)$ exists as $A$ is Noetherian.
\end{definition}

The number $ep(I)$ measures of the “thickness” of $I$. The lemma says it can be computed locally, and in the completion.

\begin{lemma}\label{eplocal}
    (i) ep(I)= $\mathrm{sup}_x(ep(I_x))$, where $x$ ranges through all closed points of $\mathrm{Spec}(A)$, and $I_x$ denotes the localisation of $I$ at $x$.\\
    (ii) Assume A is Noetherian local excellent, then $ep(I)=ep(\hat{I})$, where $\hat{I}$ is the completion of I with respect to the maximal ideal. 
\end{lemma}

\begin{proof} 
    (i) This follows from two standard commutative algebra facts: (a) localisation commutes with taking radicals; (b) inclusion relations of ideals can be checked by localisations at all closed points.\\
    
    (ii) First note that $\hat{A}$ is Noetherian (\cite[\nopp 23.K]{matsumura_commutative_1980}), and $A$ injects into $\hat{A}$ because its topology is Hausdorff ((11.D) in \textit{loc. cit.}). $A/\sqrt{I}$ being reduced implies, by excellence of $A$, $\hat{A}/\sqihat=\widehat{A/\sqrt{I}}$ is reduced, so $\sqihat$ is radical. But $\sqihat$ is contained in $\sqrt{\hat{I}}$, because if $x_{\infty}\in \sqihat$, let $\{x_i\}\subseteq \sqrt{I}$ converge to it, then $\{x_i^{ep(I)}\}\subseteq I$ converges to $x_{\infty}^{ep(I)}$, so $x_{\infty}^{ep(I)}$ lies in $\hat{I}$. So $\sqihat=\sqrt{\hat{I}}$. This argument also shows $ep(\hat{I})\leq ep(I)$. For the converse, notice $\sqrt{I}^{ep(\hat{I})}\subseteq (\sqihat)^{ep(\hat{I})}= \sqrt{\hat{I}}^{ep(\hat{I})}\subseteq \hat{I}$, so $\sqrt{I}^{ep(\hat{I})}\subseteq \hat{I}\cap A=I$, where the last step is because $\hat{I}\cap A$ is the closure of $I$ in $A$ and ideals of $A$ are closed in $A$ ((24.A) in \textit{loc. cit.}).
\end{proof}

\begin{lemma}\label{blowupmult}
    Let $C$ be a curve on a smooth surface $X$. Suppose $C$ is singular at a closed point $x$. Let $\mathrm{mult}_x(C)$ be the multiplicity of $C$ at $x$ (see, e.g., \cite[\nopp 1.20]{kollar_lectures_2007}). Then, after $M$ stages of blowups (Terminology \ref{BUstages}), the largest multiplicity of $\overline{X}\times_XC$ at its singularities is at most $2^{M-1}\cdot \mathrm{mult}_x(C)$. Here $\overline{X}$ is the result after blowups. Note $\overline{X}\times_XC$ includes the exceptional divisors together with their multiplicities.
\end{lemma}

\begin{proof}
   After the first blowup $X\leftarrow X_1$, the strict transform $C_1$ of $C$ has $\mathrm{mult}_{\overline{x}}(C_1)\leq\mathrm{mult}_x(C)$ at each point $\overline{x}$ above $x$, and the exceptional divisor $E_1$ is contained in $X_1\times_XC$ with multiplicity $\mathrm{mult}(E_1, X_1\times_XC)=\mathrm{mult}_x(C)$ (see, e.g., \cite[\nopp 1.38, 1.40]{kollar_lectures_2007}). In the second blowup (at a closed point in $E_1$), $C_2$ still has multiplicities $\leq \mathrm{mult}_x(C)$, and $\mathrm{mult}(E_2, X_2\times_XC) \leq\mathrm{max}_{\{\overline{x}\}}\{\mathrm{mult}_{\overline{x}}(C_1)\}+\mathrm{mult}(E_1, X_1\times_XC)\leq 2\cdot \mathrm{mult}_x(C)$. Iterate. 
\end{proof}

\begin{lemma}\label{epcomputation}
    Let X be a smooth surface, $\sigma\neq \mathrm{id}$ an automorphism of X, $X^{\sigma}_{red}$ the fixed locus, and $x_0\in X^{\sigma}_{red}$ a closed point. Assume $X^{\sigma}_{red}$ is an irreducible divisor smooth at $x_0$. Let $X\xleftarrow{a_1} X_1\xleftarrow{a_2}X_2\xleftarrow{a_3}\cdots\xleftarrow{a_n}X_n$ be a blowup sequence of $M$ stages, such that $a_1$ is the blowup at $x_0$, and each $a_i$ is a blowup at closed points fixed by (extensions of) $\sigma$. Then, $ep(I_{\tilde{\sigma}})\leq (2p+1)^M\cdot ep(I_{\sigma})$, where $\tilde{\sigma}$ denotes the extension of $\sigma$ to $X_n$.
\end{lemma}

Here, the extension of $\sigma$ to blowups is induced by its natural action on the tangent spaces. In the statement and proof of this lemma, we suppress the space in the notation for the ideal of fixed points. For example, in the statement, $I_{\tilde{\sigma}}$ is the abbreviation for $I_{\tilde{\sigma}, X_n}$. We will implicitly use Lemma \ref{eplocal} in the proof.

\begin{proof}
     Clearly, it suffices to show the case when the blowup sequence is monotone (Terminology \ref{BUstages}). Note $\sigma$ induces an automorphism of the Zariski localisation of $X$ at $x_0$, so we may assume $X=\mathrm{Spec}(A)$ with $A$ local. We need to analyse all possible configurations of fixed points in the successive blowups.\\
     
     The configuration we start with is $X^{\sigma}_{red}=$ a smooth curve. The blowup replaces $x_0$ by an exceptional divisor isomorphic to $\mathbb{P}^1$. Denote the extension of $\sigma$ by $\overline{\sigma}$. Then $\overline{\sigma}$ acts linearly on $\mathbb{P}^1$ (which is just the derivative of $\sigma$ at $x_0$). There are three possibilities\footnote{Indeed, let $\overline{\sigma}|_{\mathbb{P}^1}$ be represented by $[u:v]\mapsto [\alpha u+\beta v: \gamma u+\delta v]$. As $\sigma$ is an automorphism, $\begin{pmatrix}\alpha & \beta\\ \gamma & \delta \end{pmatrix}$ is in $\mathrm{GL}_2(k)$. The fixed points correspond to the eigenvectors of this matrix, and the three possibilities above correspond to the three possibilities of its Jordan normal form.}: (a) $\overline{\sigma}|_{\mathbb{P}^1}$ fixes $\mathbb{P}^1$; (b) $\overline{\sigma}|_{\mathbb{P}^1}$ fixes two points, each with multiplicity 1; (c) $\overline{\sigma}|_{\mathbb{P}^1}$ fixes a single point with multiplicity 2. Then perform the second blowup, and so on. In Figure \ref{fig:fixedptconfig}, we draw all possible local configuration changes under blowups. Solid lines and points represent fixed points. The dotted arrows represent the point of blowup. Each crossing is a simple normal crossing. One can analyse the change of $ep(I_{\sigma})$ in all cases and find the desired estimate. We illustrate with two cases, the others are similar.

\begin{figure}[H]
\centering
\includegraphics[width=0.8\textwidth]{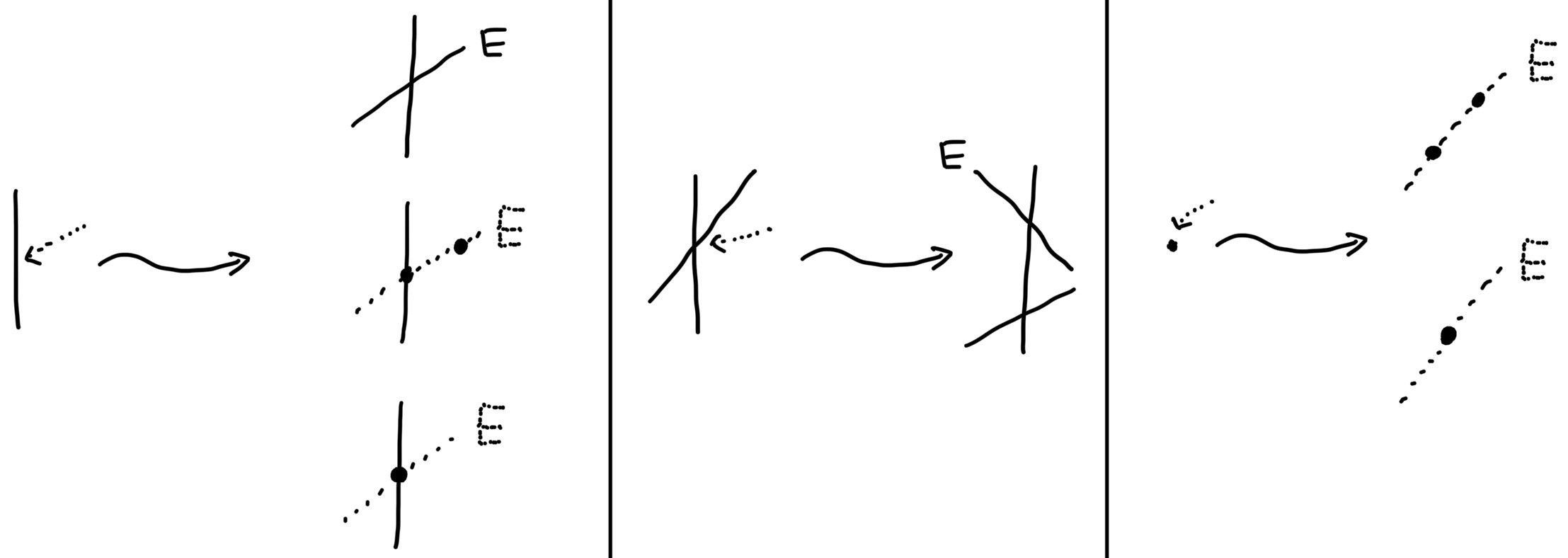}
\caption{Possible configuration changes of fixed points\label{fig:fixedptconfig}}
\end{figure}

     To separate notations, in the following we will use $W, w, \sigma$, etc. to denote the starting space, point, action, etc. and $\overline{W},\overline{w},\overline{\sigma}$, etc. to denote those after one blowup. We abbreviate $ep(I_{\sigma}), ep(I_{\overline{\sigma}})$ as $ep, \overline{ep}$. Let $\pi$ be the projection  $\overline{W}\rightarrow W$.\\
     
     General setup and observations: choose local coordinates $(x,y)$ at $w$, use coordinates $((x,y),[u:v])$ on $\overline{W}$. If $W=\mathrm{Spec}(A)$ then $\overline{W}$ is $\mathrm{Spec}(A[\frac{x}{y}])$ in the $(y,u)$-chart and is $\mathrm{Spec}(A[\frac{y}{x}])$ in the $(x,v)$-chart. Denote $\sigma^*$ the corresponding action on $A$. Then, essentially by definition, $I_{\sigma}=(\sigma^*(a)-a)$ where $a$ ranges through elements in $A$. Write $\sigma^*(x)=x+f$ and $\sigma^*(y)=y+g$, for some $f,g\in I_{\sigma}$. Note $\pi^*(I_{\sigma})\subseteq I_{\overline{\sigma}}$.\footnote{This is one precise sense how the blowup “improves the situation”.}\\
     
     Case (i). Line to cross (left column top): we may assume $W^{\sigma}_{red}=\{x=0\}$. We have $\sqrt{I_{\sigma}}=(x)$ and $x^{ep}\in I_{\sigma}$. We look at $(y,u)$-chart, the other chart is similar. Denote by $I_{\overline{\sigma},(y,u)}$ the restriction of $I_{\overline{\sigma}}$ to the $(y,u)$-chart. Then $\sqrt{I_{\overline{\sigma},(y,u)}}=(yu)$. Since $\pi^*(I_{\sigma})\subseteq I_{\overline{\sigma}}$, $x^{ep}$ lies in $I_{\overline{\sigma},(y,u)}$. But in the $(y,u)$-chart $x=yu$, so $(yu)^{ep}\in I_{\overline{\sigma}}$, hence $\overline{ep}_{(y,u)}\leq ep$.\\
     
     Case (ii). Point to point (right column bottom): we have $\sqrt{I_{\sigma}}=(x,y)$ and $(x,y)^{ep}\subseteq I_{\sigma}$. Without loss of generality, assume that the new fixed point $\overline{w}$ lies in the $(y,u)$-chart, with coordinate $(0,u_1)$. Consider the completion of $\overline{W}$ at $\overline{w}$. Denote by $I_{\overline{\sigma},\hat{w}}$ the completion of $I_{\overline{\sigma}}$ at $\overline{w}$. Then $\sqrt{I_{\overline{\sigma},\hat{w}}}=(y,u-u_1)$.\\
     
     We now exhibit some elements in $I_{\overline{\sigma},\hat{w}}$, which will be sufficient for estimating $\overline{ep}$. Since $\overline{\sigma}^*(u)=\overline{\sigma}^*(x/y)=\frac{x+f}{y+g}=\frac{u+f/y}{1+g/y}$, we have $\overline{\sigma}^*(u)-u=\frac{f/y-ug/y}{1+g/y}\in I_{\overline{\sigma},(y,u)}$. Upon localising to $\overline{w}$, $1+g/y$ becomes a unit, so $f/y-ug/y\in I_{\overline{\sigma},\hat{w}}$. On the other hand, we know $(x,y)^{ep}\subseteq I_{\sigma}$ and $\pi^*(I_{\sigma})\subseteq I_{\overline{\sigma}}$. This implies $y^{ep}\in I_{\overline{\sigma},\hat{w}}$.\\
     
     Expand $f,g$ in power series in $y$ and $u'=u-u_1$, the terms in $f/y-ug/y$ which do not involve $y$ are precisely the terms in the quadratic equation on $\mathbb{P}^1$ for the fixed point $\overline{w}$. So $f/y-ug/y\in k[[y,u']]$ is of the form $au'^2$ or $au'^2+y(...)$ with $(...)\neq 0$, for some $a\neq 0$ in $k$. It is an exercise to see that in the former case $\overline{ep}\leq ep+2$, and in the latter case $\overline{ep}\leq (2p+1)\cdot ep$.\footnote{Proof in the latter case: $\forall p^m\in\mathbb{N}$, $(au'^2+y(...))^{p^m}=a^{p^m}u'^{2p^m}+y^{p^m}(...)^{p^m}\in I_{\overline{\sigma},\hat{w}}$. If $p^m$ is the smallest integer $\geq ep$ (which is clearly $< p\cdot ep$), then, using $y^{ep}\in I_{\overline{\sigma},\hat{w}}$, we get $u'^{2p^m}\in I_{\overline{\sigma},\hat{w}}$, so $\overline{ep}\leq ep+2p^m$.}
\end{proof}

The aim of the rest of this subsection is Proposition \ref{prop_boundonBUstages}. We first fix the setup and have some preliminary discussion. Let $X$ be a smooth surface, $D$ a divisor which is reduced and irreducible, $U=X-D$ the complement, $U'\rightarrow U$ a Galois covering with Galois group $G$, $\pi: X'\rightarrow X$ the normalisation of $X$ in $U'$, $D'=X'\times_X D$ the pullback, and $x\in D$ a closed point. Note $\pi$ is finite. The following lemma is standard, we include a proof for completeness.

\begin{lemma}
    With the above setup, let $X_{\hat{x}}$ be the completion of $X$ at $x$, and $X'_{\hat{x}}=X'\times_XX_{\hat{x}}$. Then $X'_{\hat{x}}\cong\bigsqcup_{x'\in\pi^{-1}(x)}X'_{\hat{x}'}$, where $X'_{\hat{x}'}$ is the completion of $X'$ at $x'$. Moreover, for each $x'$, $X'_{\hat{x}'}\rightarrow X_{\hat{x}}$ is the normalisation of $X_{\hat{x}}$ in the function field $K(X'_{\hat{x}'})$, which is a finite separable extension of $K(X_{\hat{x}})$.
\end{lemma}

\begin{proof}
    Let $X_{x}=\mathrm{Spec}(A), X_{\hat{x}}=\mathrm{Spec}(\hat{A})$. Then $ X'\times_X X_{x}=\mathrm{Spec}(A'_L)$, where $A'_L$ is the normalisation of $A$ in the finite separable extension $L=K(X')$ of $K(X)=Q(A)$ ($Q$ denotes the total ring of fractions). As $A$ is Japanese,  $A'_L$ is a finite $A$-module, so $A'_L\otimes_A \hat{A}\cong \hat{A}'_L$ by \cite[\nopp 23.L, Theorem 55]{matsumura_commutative_1980}, where $\hat{A}'_L$ is the completion of $A'_L$ as an $A$-algebra at $\mathfrak{m}_A$. $\hat{A}'_L$, being finite over a local ring, is semi-local, and its maximal ideals are in bijection with $\pi^{-1}(x)$. Apply (24.C) in \textit{loc. cit.}, we get $\hat{A}'_L\cong \Pi_{x'\in \pi^{-1}(x)} (\hat{A}'_L)_{\hat{x}'}$ (the product of completions of $\hat{A}'_L$ at $x'$)\footnote{Note $V(\mathfrak{m}_{\hat{A}'_L})=V(\mathfrak{m}_{\hat{A}}\hat{A}'_L)$, so $\hat{A}'_L$ is already complete with respect to $\mathfrak{m}_{\hat{A}'_L}$.}. This proves the first statement.\\

    For the second statement, note we have the inclusions $\hat{A}\subseteq \hat{A}'_L\subseteq Q(\hat{A}'_L)$. It suffices to show $\hat{A}'_L$ is normal. In fact, $\hat{A}'_L$ is finite hence integral over $\hat{A}$, if it is normal, then it is the integral closure of $\hat{A}$ in $Q(\hat{A}'_L)$. Further, $Q(\hat{A}'_L)\cong \Pi_{x'\in \pi^{-1}(x)}Q((\hat{A}'_L)_{\hat{x}'})$, so for each $x'$, $(\hat{A}'_L)_{\hat{x}'}$ is the integral closure of $\hat{A}$ in $Q((\hat{A}'_L)_{\hat{x}'})$, which is a finite extension of $Q(\hat{A})$. The separability follows from the separability of $U'$ over $U$ and base change. To see $\hat{A}'_L$ is normal, consider 
\[\begin{tikzcd}
	{A'_L} & {\hat{A}'_L} \\
	A & {\hat{A}}
	\arrow[from=2-1, to=2-2]
	\arrow[from=2-1, to=1-1]
	\arrow[from=1-1, to=1-2]
	\arrow[from=2-2, to=1-2]
	\arrow["\lrcorner"{anchor=center, pos=0.125, rotate=-90}, draw=none, from=1-2, to=2-1]
\end{tikzcd}\]

     The formal fibres of $A'_L$ are base changes of formal fibres of $A$ under finite field extensions. Since $A$ is excellent, its formal fibres are geometrically regular, so the formal fibres of $A'_L$ are also geometrically regular. So the normality of $A'_L$ implies the normality of $\hat{A}'_L$ by (21.E) in \textit{loc. cit.}
\end{proof}

Now consider $X'_{\hat{x}'}\rightarrow X_{\hat{x}}$ for a fixed $x'$. Suppose further that $X'$ is smooth at $x'$. Denote the function rings of $X_{\hat{x}}$ and $X'_{\hat{x}'}$ by $A$ and $R$, respectively (note the notations are different from those in the above proof). Let $C=V(h), h\in \mathfrak{m}_A$ be a smooth curve in $X_{\hat{x}}$, with no common components with $D_{\hat{x}}$, where $D_{\hat{x}}:=D\times_XX_{\hat{x}}=V(g)$, for some $g\in\mathfrak{m}_A$. Denote $A'=A/hA, R'=R/hR$. Let $\mathfrak{C}(R/A)$ be the Dedekind codifferent. We refer to \cite[\nopp 0BW0 and related tags]{Stacks} for generalities on Dedekind codifferents. We summarise what we need in the following:

\begin{facts}[about Dedekind codifferents]
     Notations here are separate from the above. Let $A\rightarrow B$ be a map between Noetherian rings. Assume the map is finite, any nonzerodivisor in $A$ maps to a nonzerodivisor in $B$, and $K\rightarrow L$ is \etale, where $K:=Q(A), L:=Q(A)\otimes_A B$. Under these assumptions, $\mathfrak{C}(B/A)$ is defined to be $\{x\in L| \mathrm{Tr}_{L/K}(xb)\in A, \forall b\in B\}$. Assume further $A\rightarrow B$ is flat. We have:\\\\
     (a) $\mathfrak{C}(B/A)$ is a finite $A$-module (being the inverse of an ideal of a Noetherian ring (\cite[\nopp 0BW1]{Stacks})).\\\\
     (b) If $A\rightarrow B$ is a local complete intersection map, then $\mathfrak{C}(B/A)$ equals the inverse of the Kähler different (\cite[\nopp 0BW1, 0BW5]{Stacks}). Note, with our assumptions, the Kähler different is the same as the Jacobian ideal (see \cite[\nopp 4.4]{swanson_integral_2006} for generalities on Jacobian ideals). Both equal the ideal generated by (the image in $B$ of) $\Delta:=\mathrm{det}(\partial f_i/\partial x_j)$, where $B\cong A[x_1,...,x_n]/(f_1,...,f_n)$ is any presentation of $B$ as an $A$-algebra.
\end{facts}

\begin{proposition}\label{prop_boundonBUstages}
     With the above setup, let $M_C$ be the number of blowup stages needed to resolve $C':=X'_{\hat{x}'}\times_{X_{\hat{x}}}C\hookrightarrow X'_{\hat{x}'}$. Then $M_C\leq r\cdot s\cdot (C\cdot D_{\hat{x}})$, where $r$ is the smallest integer such that $g^rA\subseteq \mathrm{Ann}_A(\mathfrak{C}(R/A)/R)$, and $s$ is the smallest number of generators of $\mathfrak{C}(R/A)/R$ as an $A$-module.
\end{proposition}

The proof is due to Bernd Ulrich\footnote{Ulrich's original argument in fact also applies to the case without the assumption of $X'_{\hat{x}'}$ being smooth. We present this special case for the sake of simplicity. In that generality, the last paragraph needs to be modified: instead of Fact (b), one applies \cite[\nopp 3.5]{kunz_regular_1988} to compute $\mathfrak{C}(R/A)$ (resp. $\mathfrak{C}(R'/A')$), and get $\Delta^{-1}((f_1,...,f_n):I)$ (resp. $\tilde{\Delta}^{-1}((\tilde{f}_1,...,\tilde{f}_n):\tilde{I})$), where $R\cong A[x_1,...,x_n]/I$ is a presentation and $(f_1,...,f_n)$ is a regular sequence in $I$ (which exists because $R$ is Cohen-Macaulay, being normal of dimension 2).}.

\begin{proof}
    First note $M_C\leq \mathrm{dim}_k(\overline{R}'/R')$, where $\overline{R}'$ is the normalisation of $R'$, i.e., the integral closure of $R'$ in $Q(R')$. This follows from a simple application of \cite[Theorem 4, Theorem 5]{northcott_notion_1957} and the well-known fact that finitely many blowups resolve $R'$. Now consider the inclusions in $Q(R')$: $R'\subseteq \overline{R}'\subseteq \mathfrak{C}(\overline{R}'/A')\subseteq \mathfrak{C}(R'/A')$. We get $M_C\leq \mathrm{dim}_k(\overline{R}'/R')=\lambda_A(\overline{R}'/R')\leq \lambda_A(\mathfrak{C}(R'/A')/R')$.\\
    
    \underline{Claim}: $\mathfrak{C}(R'/A')= \mathfrak{C}(R/A)'(:=\mathfrak{C}(R/A)\otimes_A A'$). Consequently $\mathfrak{C}(R'/A')/R'= \mathfrak{C}(R/A)'/R'= (\mathfrak{C}(R/A)/R)'$ $(:=(\mathfrak{C}(R/A)/R)\otimes_A A')$.\\

    Assume this for now, we prove the proposition. It remains to bound $\lambda_A(\mathfrak{C}(R/A)/R)')$. Denote $N=\mathfrak{C}(R/A)/R$. $N$ is a finite $A$-module, supported in $D_{\hat{x}}$ (because the $A\rightarrow R$ is \etale elsewhere). Let $r$ and $s$ be as in the statement of the proposition. Then $N$ is naturally an $A/g^rA$-module, $N'= N\otimes_{A/g^rA}A/(g^r, h)$. We have $(A/g^rA)^s\twoheadrightarrow N$, hence $(A/(g^r, h)^s\twoheadrightarrow N'$ after tensoring with $A/(g^r, h)$. So $\lambda_A(N')\leq \lambda_A((A/(g^r, h))^s)=r\cdot s\cdot (C\cdot D_{\hat{x}})$.\\

    Finally, we prove the claim. By assumption, $A$ and $R$ are regular, so $R$ is a local complete intersection over $A$. Choose any presentation $R\cong A[x_1,...,x_n]/(f_1,...,f_n)$, where $(f_1,...,f_n)$ is a regular sequence in  $A[x_1,...,x_n]$. Then, by Fact b) above, $\mathfrak{C}(R/A)=R\Delta^{-1}$, where $\Delta$ denotes the image of $\mathrm{det}(\partial f_i/\partial x_j)$ in $R$. Base change to over $A'$, we get $R'\cong A'[x_1,...,x_n]/(\tilde{f}_1,...,\tilde{f}_n)$, where $\tilde{f}_i$ denotes the image of $f_i$ in $A'[x_1,...,x_n]$. Since $\mathrm{Spec}(R') (=C')$ is of codimension 1 in $\mathrm{Spec}(R) (=X'_{\hat{x}'})$, $(\tilde{f}_1,...,\tilde{f}_n)$ is still a regular sequence in $A'[x_1,...,x_n]$. So, by Fact b) again, $\mathfrak{C}(R'/A')=R'\tilde{\Delta}^{-1}$, where $\tilde{\Delta}$ denotes the image of $\mathrm{det}(\partial \tilde{f}_i/\partial x_j)$ in $R'$. We conclude that $\mathfrak{C}(R'/A')=(\mathfrak{C}(R/A))'$.
\end{proof}
    
\subsection{Resolution of a class of singularities}\label{subsec_resolution}
The materials in this subsection will be used in the proof of Proposition \ref{tamearemuc}.

\begin{lemma}\label{lem_resolution1}
    Let $X$ be an affine and smooth variety, $\{x_1,...,x_n\}$ an \etale coordinate system on $X$ \footnote{I.e., $\{x_1,...,x_n\}\subseteq \mathcal{O}_X(X)$ and the map $X\xrightarrow{(x_1,...,x_n)}\mathbb{A}^n$ is \etale.}such that $D=\cup_{i=1}^rD_i:=\cup_{i=1}^r\{x_i=0\}$ is the decomposition of a simple normal crossing divisor (sncd) $D\hookrightarrow X$ into its irreducible components, $0<r\leq n$. Denote $\cap_{i=1}^rD_i$ by $\underline{D}$. Let $f$ be a ttfun of $T^*_{\underline{D}}X$ \footnote{We use the following notation: for $D=\cup_{i=1}^rD_i\hookrightarrow X$ a sncd, $T^*_DX:=\cup_I T^*_{D_I}X$, where $I$ ranges through nonempty subsets of $\{1,2,...,r\}$, and $D_I:=\cap_{i\in I}D_i$.\label{footnoteonnotation}} at $(x,\xi)$ where $x=\mathrm{origin}$, and $\xi=dx_1+...+dx_r$.\footnote{For $C$ a conical closed subset of $T^*X$ and $(x,\xi)$ a smooth point of $C$, we say $f$ is a ttfun  of $C$ at $(x,\xi)$ if $f$ satisfies the same conditions as in Definition \ref{ttfun}, with “$SS\mathcal{F}$” replaced by “$C$”.} Denote $f^{-1}(0)$ by $H$. Then, $D_1\cap H,...,D_{r-1}\cap H$ form a sncd on $H$ and $x_r|_H$ is a ttfun of $T^*_{D_1\cap...\cap D_{r-1}\cap H}H$ at $(x,-(dx_1+...+dx_{r-1})|_H)$. (For $r=1$, $T^*_{D_1\cap...\cap D_{r-1}\cap H}H:=T^*_HH$.)
\end{lemma}

\begin{proof}
    It follows easily from $f$ being a ttfun that:\\
    (i) $D_1,...,D_{r-1}$ indeed form a sncd on $H$;\\
    (ii) $\Gamma_{dx_r|_H}$ intersects $T^*_{D_1\cap...\cap D_{r-1}\cap H}H$ precisely at $(x,-(dx_1+...+dx_{r-1})|_H)$.\\
    
    We want to show the intersection is transverse. By the “$\Rightarrow$” part of the proof of Proposition \ref{misclemmacloimm}, it suffices to show this in the ambient space $X$, i.e., that \\$\Gamma_{dx_r}\cdot \langle dx_1,...,dx_{r-1}, df\rangle=1\cdot (x, dx_r)$, where $\langle dx_1,...,dx_{r-1}, df\rangle$ denotes the pushforward of $T^*_{D_1\cap...\cap D_{r-1}\cap H}H$ into $X$. The $r=n$ case is easy. We assume $r\leq n-1$ in the following. Note $\Gamma_{dx_r}$ and $\langle dx_1,...,dx_{r-1}, df\rangle$ intersect precisely at $(x, dx_r)$, so it suffices to show that their tangents at $(x, dx_r)$ are linearly independent. The computation is straightforward, here are the results:\\\\
    Use coordinates $\{x_1,...,x_n;p_1,...,p_n\}$ on $T^*X$.\\
    $\Gamma_{dx_r}$: tangent space at $(x, dx_r)$ is spanned by $\{\partial_{x_1},...,\partial_{x_n}\}$.\\
    $\langle dx_1,...,dx_{r-1}, df\rangle$: tangent space  at $(x, dx_r)$ is spanned by $\{\partial_{p_1},...,\partial_{p_{r-1}}, \sum _{i=1}^n \partial_{p_i},$\\$ (\partial_{x_{r+1}}+\sum_{i=1}^n f_{i,r+1}\partial_{p_i}),...,(\partial_{x_{n}}+\sum_{i=1}^n f_{i,n}\partial_{p_i})\}$, where $f_{i,j}$ denotes the derivative of $f$ in $x_i$ followed by in $x_j$.\\
    
    These are linearly independent if and only if the matrix $\{f_{i,j}\}_{i,j\in \{r+1,...,n\}}$ is nondegenerate. But this follows exactly from the assumption that $f$ is a ttfun for $T^*_{\underline{D}}X$.
\end{proof}

\begin{lemma}\label{lem_resolution2}
    Same setup as in the previous lemma. Then the embedded singularity $D\cap H \hookrightarrow H$ can be resolved in two steps: first blow up at $x$, then blow up along the intersection of the exceptional divisor with the strict transform of $D_1\cap ... \cap D_{r-1}\cap H$. (There is only one blowup for $r=1$.)
\end{lemma}

\begin{proof}
    As everything happens on $H$, for convenience, we make the following notation changes in this proof (new $\dashleftarrow$ old): $X\dashleftarrow H$, $D_i\dashleftarrow D_i\cap H$, $x_i\dashleftarrow x_i|_H$,  (both for $i=1,2,...,r-1$), $D\dashleftarrow \cup_{i=1}^{r-1}D_i\cap H$, $\underline{D}\dashleftarrow \cap_{i=1}^{r-1}D_i\cap H$, $H\dashleftarrow D_r\cap H$, $f\dashleftarrow x_r|_H$. We also rename $n$ and $r$ such that our new $X$ is of dimension $n$, and new $D$ has $r$ components. In this new notation, the statement becomes: the embedded singularity $D\cup H \hookrightarrow X$ can be resolved by first blowing up at $x$, then blowing up along the intersection of the exceptional divisor with the strict transform of $\underline{D}$.\\
    
    The $r=n$ case is simple. We check the $r\leq n-1$ cases. The condition on $f$ implies that, when pulled back to the completion $X_{\hat{x}}=\mathrm{Spec}(k[[x_1,...,x_n]])$ of $X$ at $x$, $f$ is of the form $x_1+...+x_r+Q+Q'+P+(...)$ where $Q$ is a nondegenerate quadratic form in $\{x_{r+1},...,x_{n}\}$, $Q'$ is a quadratic form in $\{x_1,...,x_r\}$, $P$ is a linear combination of monomials of the form $x_ax_{\alpha}$ for $a\in \{1,...,r\}$, $\alpha\in \{r+1,...,n\}$, and $(...)$ means higher degree terms. In the rest of the proof, $a,b...$ always mean an index in $\{1,...,r\}$;  $\alpha,\beta...$ always mean an index in $\{r+1,...,n\}$; $i,j...$ always mean an index in $\{1,...,n\}$; $\sum$ means over all allowed indices unless specified. By a linear change of coordinates in $\{x_{\alpha}\}$ and possibly replacing $X$ by a smaller affine open neighbourhood of $x$, we may assume $Q=\sum x_{\alpha}^2$.\\
    
    Blow up at $x$. Use coordinates $((x_1,...,x_n),[p_1:...:p_n])$ on $X\times \mathbb{P}^{n-1}$. Then  $\mathrm{Bl}_x(X)$ is the subvariety cut out by $\{x_ip_j=x_jp_i\}_{\mathrm{all}\, i,j}$. We look at the $p_n=1$ piece, the others can be checked by the same method. On this piece, we may use coordinates $\{x_n,p_1,...,p_{n-1}\}$, as $x_i=x_np_i$ for $i=1,...,n-1$. The exceptional divisor is $E=\{x_n=0\}$. We list the strict transforms of relevant objects:\\
    $\bullet$ $D_a\leadsto D_a'=\{p_a=0\}$;\\  
    $\bullet$ $\underline{D}\leadsto\underline{D}'=\{p_1=...=p_r=0\}$;\\
    $\bullet$ $H\leadsto H'=\{f^{(1)}=0\}$, where $f^{(1)}$ is a regular function on $\mathrm{Bl}_x(X)_{p_n\neq 0}$ whose pullback to $\mathrm{Bl}_{\hat{x}}(X_{\hat{x}})_{p_n\neq 0}$ is $\sum p_a+(x_n+\sum_{\alpha=r+1}^{n-1} x_np_{\alpha}^2)+Q'/x_n+P/x_n+(...)$, with $Q'/x_n$ consisting of (linear combinations of) terms of the form $x_np_ap_b$, and $P/x_n$ of terms of the forms $x_np_ap_{\alpha}$ and $x_np_a$. \\
    
    The conormals of $D_a', \underline{D}'$ are spanned by $dp_a, \{dp_1,...,dp_r\}$ respectively. We compute: $df^{(1)}|_E=\sum dp_a+ (1+\sum_{\alpha=r+1}^{n-1} p_{\alpha}^2)dx_n+A$, where $A$ consists terms of the form $p_ap_bdx_n, p_ap_{\alpha}dx_n, p_adx_n$. It is an exercise to deduce from these that:\\
    \noindent
    (i) $\{D_1',...,D_r', H', E\}$ form a sncd except along $\underline{D}'\cap H'\cap E=\underline{D}'\cap E$;\\
    (ii) along $\underline{D}'\cap E$: outside the conic $C:=\{x_n=p_1=...=p_r=0, 1+\sum_{\alpha=r+1}^{n-1} p_{\alpha}^2=0\}$, every $r-1$ members of $\{D_1',...,D_r', H', E\}$ form a sncd; on $C$, $\{D_1',...,D_r', E\}$ form a sncd, and $df^{(1)}=\sum dp_a$.\\
    
    It remains to resolve the singularity along $\underline{D}'\cap E$. Blow up along $\underline{D}'\cap E$. It is another exercise to see, using (ii), that the singularity outside $C$ is resolved. We check that the singularity is also resolved over $C$.\\
    
    Use coordinates $((x_n,p_1,...,p_{n-1}),[q_n:q_1:...:q_r])$ on $(\mathrm{Bl}_{x}(X_{x})_{p_n\neq 0})\times\mathbb{P}^{r}$. Then $\mathrm{Bl}_{\underline{D}'\cap E}(\mathrm{Bl}_{x}(X_{x})_{p_n\neq 0})$ is the subvariety cut out by $\{x_nq_i=q_np_i, p_iq_j=q_ip_j\}_{\mathrm{all}\, i,j}$. We look at the $q_n=1$ piece, the others can be checked by the same method. On this piece, we may use coordinates $\{x_n,q_1,...,q_r,p_{r+1},...,p_{n-1}\}$, as $p_a=x_nq_a$ for $a=1,...,r$. The exceptional divisor is $F=\{x_n=0\}$. We list the strict transforms of relevant objects:\\
    $\bullet$ the strict transform of $E$ lies at infinity and is irrelevant on this piece;\\
    $\bullet$ $D_a'\leadsto D_a''=\{q_a=0\}$;\\
    $\bullet$ $\underline{D}'\leadsto \underline{D}''=\{q_1=...=q_r=0\}$;\\
    $\bullet$ $H'\leadsto H''=\{f^{(2)}=0\}$, where $f^{(2)}$ is a regular function on $(\mathrm{Bl}_{\underline{D}'\cap E}(\mathrm{Bl}_{x}(X_{x})_{p_n\neq 0}))_{q_n\neq 0}$ whose pullback to $(\mathrm{Bl}_{\underline{D}'\cap E}(\mathrm{Bl}_{\hat{x}}(X_{\hat{x}})_{p_n\neq 0}))_{q_n\neq 0}$ is $\sum q_a+(1+\sum_{\alpha=r+1}^{n-1} p_{\alpha}^2)+Q'/x^2_n+P/x^2_n+(...)$, with $Q'/x^2_n$ consisting of terms of the form $x_n^2q_aq_b$, and $P/x^2_n$ of terms of the forms $x_nq_ap_{\alpha}$ and $x_nq_a$.\\
    
    We compute: $df^{(2)}|_{x_n=0}=\sum dq_a+(\sum_{\alpha=r+1}^{n-1} 2p_{\alpha}dp_{\alpha})+A$, where $A$ consists of terms of the form $q_ap_{\alpha}dx_n, q_adx_n$. Recall that we want to show $\{D_1'',...,D_r'', H'', F, E'\}$ form a sncd. $E'$ is irrelevant here, and it suffices to check over $C$, i.e., on the locus $\{x_n=0, 1+\sum_{\alpha=r+1}^{n-1} p_{\alpha}^2=0\}$. 
    But along this locus, $\{p_{r+1},...,p_{n-1}\}$ are not all 0, so $df^{(2)}|_{x_n=0}$ contains some $dp_{\alpha}$ component and is thus not contained in the span of $\{dx_n,dq_1,...,dq_r\}$, consequently $\{D_1'',...,D_r'', H'', F, E'\}$ form a sncd. 
\end{proof}

\subsection{A retraction lemma}
We learned the following lemma from Owen Barrett. This will be used in the proof of Proposition \ref{misclemmacloimm}.

\begin{lemma}\label{owenretraction}
    Let $i: Z\hookrightarrow X$ be a closed immersion of smooth schemes over a field $k$, and $z\in Z$ a point. Then in some Zariski open neighbourhood $X'$ of $z$ in $X$, there exists an \etale neighbourhood $\Tilde{X'}\xrightarrow{\beta}X'$ of $z$ in $X'$ and maps $\alpha, r$ making the following diagram commute:
\[\begin{tikzcd}
	{Z'} & {\tilde{X'}} & {X'}
	\arrow["\alpha"', hook, from=1-1, to=1-2]
	\arrow["\beta"', from=1-2, to=1-3]
	\arrow["r"', curve={height=12pt}, from=1-2, to=1-1]
\end{tikzcd}\]
where $Z'=Z\times_X X'$, $\alpha$ is a closed immersion, $\beta$ is \etale, $r$ is a retraction, and $\beta\alpha=i$. Moreover, the retraction $r$ is smooth along $Z'$.
\end{lemma}

\begin{proof}
    As $Z$ is smooth, there exists an \etale map $Z'\rightarrow \mathbb{A}_k^m$ for some Zariski open neighbourhood $Z'$ of $z$ in $Z$. Extend $Z'$ to a Zariski open $X'$ in $X$. By further shrinking, we may assume $Z', X'$ are affine. Consider the following pushout, and choose a retraction $r'$:
\[\begin{tikzcd}
	{Z'} & {X'} \\
	{\mathbb{A}^m_k} & {X''}
	\arrow[hook, from=1-1, to=1-2]
	\arrow[from=1-1, to=2-1]
	\arrow["f", from=1-2, to=2-2]
	\arrow[hook, from=2-1, to=2-2]
	\arrow["\lrcorner"{anchor=center, pos=0.125, rotate=180}, draw=none, from=2-2, to=1-1]
	\arrow["{r'}"{description}, curve={height=-12pt}, from=2-2, to=2-1]
\end{tikzcd}\]
(If $X''=\mathrm{Spec}A, \mathbb{A}^m_k=\mathrm{Spec}(k[x_1,...,x_m])$, choosing an $r'$ amounts to choosing a lift for each $x_i$ of $A\twoheadrightarrow k[x_1,...,x_m]$.) We construct $\Tilde{X'}$ etc. using the following two pullback diagrams:
\[\begin{tikzcd}
	{Z'} & {Z'\times_{\mathbb{A}^m_k}X'} && {Z'\times_{\mathbb{A}^m_k}Z'} & {Z'\times_{\mathbb{A}^m_k}X'} \\
	{\mathbb{A}^m_k} & {X'} && {Z'} & {X'}
	\arrow["{pr_2}", from=1-5, to=2-5]
	\arrow["i"', hook, from=2-4, to=2-5]
	\arrow["{pr_2}", from=1-4, to=2-4]
	\arrow["{i'}", hook, from=1-4, to=1-5]
	\arrow["\lrcorner"{anchor=center, pos=0.125}, draw=none, from=1-4, to=2-5]
	\arrow["\Delta", curve={height=-12pt}, from=2-4, to=1-4]
	\arrow[from=1-1, to=2-1]
	\arrow["{r'f}", from=2-2, to=2-1]
	\arrow["{pr_2}", from=1-2, to=2-2]
	\arrow["{pr_1}"', from=1-2, to=1-1]
	\arrow["\lrcorner"{anchor=center, pos=0.125, rotate=-90}, draw=none, from=1-2, to=2-1]
\end{tikzcd}\]

In the right diagram, note $Z'\times_{\mathbb{A}^m_k}Z'$ is a disjoint union of several copies of $Z'$ because $pr_2$ is \etale. $\Delta$ is an isomorphism to the diagonal copy. Let $\Tilde{X'}=Z'\times_{\mathbb{A}^m_k}X'-(Z'\times_{\mathbb{A}^m_k}Z'-\Delta(Z')), \alpha=i'\Delta$, $\beta=pr_2$, $r=pr_1$. It is an exercise to see that these satisfy the requirement. Note the smoothness of $r$ along $Z'$ follows from the smoothness of $Z', \tilde{X}'$ and the injectivity of $dr$ on cotangent spaces (see, e.g., \cite[\nopp 6.2.10]{liu_algebraic_2002}).
\end{proof}

\section{The stability of vanishing cycles}\label{stab of phi}
This section is devoted to discussing the stability of vanishing cycles in the positive characteristic algebraic context. In §\ref{subsection-stabofdimtot} we discuss the independence of dimtot($\phi$) with respect to the ttfun. In §\ref{subsection on high-jet} we discuss the variation of $\phi$ with respect to high jets of the ttfun. We fix the following setup: $X$ is a smooth variety over a field $k$ algebraically closed of characteristic $p>2$, $\mathcal{F}\in D(X)$, and $(x,\xi)$ is a \emph{smooth} point of $SS\mathcal{F}$. Note that in this setup ttfun's at $(x,\xi)$ always exist (\cite[\nopp 4.12]{beilinson_constructible_2016}).

\subsection{The stability of dimtot($\phi$)}\label{subsection-stabofdimtot}

\begin{proposition}[\cite{saito_characteristic_2017}]\label{dimtotstab}
    With the above setup, assume further $\mathcal{F}\in D_{ctf}(X)$, then for a ttfun $f$, $\mathrm{dimtot(}\phi_f(\mathcal{F})_x\mathrm{)}$ is independent of $f$.
\end{proposition}

To see this is true, just apply the Milnor formula (\cite[\nopp 5.9]{saito_characteristic_2017}): for a ttfun $f$, $\mathrm{dimtot(}\phi_f(\mathcal{F})_x\mathrm{)}$ is equal to minus the coefficient of $CC\mathcal{F}$ at $x$. But logically, this proposition comes before the Milnor formula. Indeed, the very fact that the dimtot of vanishing cycles “lives” on the cotangent bundle allows one to define the characteristic cycle. See Remark \ref{dimtotstabproof} for a logically direct proof of this proposition.
\begin{proposition}[\cite{saito_characteristic_2017}] \label{phistabradon}
    The integer $\mathrm{dimtot(}\phi\mathrm{)}$ is invariant under the Radon transform. More precisely, in the algebraic Radon setup \ref{radonsetup}, assume further $\mathcal{F}\in D_{ctf}(\mathbb{P})$, let $(x,\xi)$ be a smooth point of $SS\mathcal{F}$ with $\xi\neq 0$. Denote by $\nu$ the image of $(x,\xi)$ in $PSS\mathcal{F}$. Let $(a, \alpha)$ be any representative of $\nu$ in $T^*\mathbb{P}^{\vee}$. Let $f$ (resp. $g$) be any ttfun for $\mathcal{F}$ (resp. $R\mathcal{F}$) at $(x,\xi)$ (resp. $(a,\alpha)$). Then $\mathrm{dimtot(}\phi_f(\mathcal{F})_x\mathrm{)} = \mathrm{dimtot(}\phi_g(R\mathcal{F})_a\mathrm{)}$.
\end{proposition}
 
This is a consequence of \cite[\nopp 6.4, 6.5]{saito_characteristic_2017}. For later purposes, we record a direct proof in our setting.
\begin{proof}
    By the compatibility of vanishing cycles with proper pushforwards, $\phi_g(\q_!\p^*\mathcal{F})_a\cong \q_*\phi_{g\q}(\p^*\mathcal{F})$. By the following Lemma \ref{radonintersect}, $g\q$ is a ttfun for $\p^*\mathcal{F}$ at $(z, \zeta)$, where $z=(x,a), \zeta=d\q(\alpha)$. Moreover, $z$ is the only point in $\q^{-1}(a)$ where $d(g\q)$ and $SS(\p^*\mathcal{F})$ intersects. So $\q_*\phi_{g\q}(\p^*\mathcal{F})\cong \phi_{g\q}(\p^*\mathcal{F})_z$ as objects in $D^b_c(\mathbb{Z}/\ell^n[G_\eta])$ (see Conventions for notations). By Proposition \ref{dimtotstab}, dimtot($\phi_{g\q}(\p^*\mathcal{F})_z$) can be computed by any ttfun for $\p^*\mathcal{F}$ at $(z, \zeta)$. Locally near $x\in\mathbb{P}$, fix an isomorphism of $Q\xrightarrow{\p}\mathbb{P}$ with the projection  $\mathbb{P}\times\mathbb{A}^{n-1}\rightarrow\mathbb{P}$ such that $z$ is identified with $(x,0)$. Let $\{y_1,...,y_{n-1}\}$ be the standard coordinate on $\mathbb{A}^{n-1}$. Then, for $h=y_1^2+...+y_{n-1}^2$, $f+h$ is a ttfun at $(z,\zeta)$. Apply Thom-Sebastiani \cite[\nopp 4.5]{illusie_around_2017}, we get $\phi_f(\mathcal{F})_x\ast\phi_h(\underline{\mathbb{Z}/\ell^n})_0\simeq\phi_{f+h}(\p^*\mathcal{F})_z$. To show $\mathrm{dimtot(}\phi_f(\mathcal{F})_x\mathrm{)} = \mathrm{dimtot(}\phi_g(R\mathcal{F})_a\mathrm{)}$, it suffices to show $\mathrm{dimtot(}\phi_f(\mathcal{F})_x\mathrm{)} = (-1)^{n-1}\mathrm{dimtot(}\phi_{f+h}(\p^*\mathcal{F})_z\mathrm{)}$, where the $n-1$ comes from the definition of the Radon transform. By the perverse t-exactness of shifted vanishing cycles, we may assume $\mathcal{F}$ is perverse. Then $\phi_f(\mathcal{F})_x$ is in degree $0$, and $\phi_{f+h}(\p^*\mathcal{F})_z$ is in degree $n-2$. It is well-known that $\phi_h(\underline{\mathbb{Z}/\ell^n})_0$ is in degree $n-2$ and of rank $1$ (\cite[\nopp XV 2.2.5]{SGA7}). Apply  \cite[\nopp 5.12]{illusie_around_2017}, we get $\mathrm{dimtot(}\phi_{f+h}(\p^*\mathcal{F})_z\mathrm{)}=(-1)^{n-1}\mathrm{dimtot(}\phi_{f}(\mathcal{F})_x\mathrm{)}$.
\end{proof}

\begin{lemma}\label{radonintersect}
In the algebraic Radon setup \ref{radonsetup}, let $\mathcal{F}\in D(\mathbb{P})$, $(a,\alpha)$ a smooth point in $SSR\mathcal{F}$, $\alpha\neq 0$, and $g$ a ttfun for $R\mathcal{F}$ at $(a,\alpha)$. Then:\\
(i) on $\q^{-1}(a)$, $\Gamma_{d(g\q)}$ intersects $SS(\p^*\mathcal{F})$ only at $(z, \zeta)\in T^*Q$, where $z$ is the point in $Q\cong PT^*\mathbb{P}^{\vee}$ corresponding to $(a, \alpha)$, and $\zeta=d\q(\alpha)$;\\
(ii) the intersection of $\Gamma_{d(g\q)}$ and $SS(\p^*\mathcal{F})$ at $(z, \zeta)$ is transverse.\\
In particular, $g\q$ is a ttfun for $\p^*\mathcal{F}$ at $(z, \zeta)$.
\end{lemma}
\begin{proof}
    (i) If $(z', \zeta')$ is in the intersection and $z'\in \q^{-1}(a)$, because $SS^+(R\mathcal{F})=\q_{\circ}SS^+(\p^*\mathcal{F})=\q_{\circ}\p^{\circ}SS^+\mathcal{F}$, there must exist an $(x', \xi')\in T^*\mathbb{P}$ such that (a) it lies in $SS^+\mathcal{F}$; (b) it corresponds to $(a, \alpha)$. But (a) forces $x'=\p(z')$, $\xi'=$ the unique covector at $x'$ which pulls back under $d\p_{x'}$ to $\zeta'$; (b) forces $(z', \zeta')=(z, \zeta)$.\\
    
    Note, actually more is true: if we restrict to a small Zariski neighbourhood $V$ of $a$ on which $\Gamma_{dg}$ intersects $SS^+(R\mathcal{F})$ only at $(a, \alpha)$, then $(z, \zeta)$ is the only point of intersection of $\Gamma_{d(g\q)}$ and $SS(\p^*\mathcal{F})$ on $\q^{-1}(V)$.\\
    
    (ii) Let $V$ be a neighbourhood as above. Consider the correspondence: 
\[\begin{tikzcd}
	& {Q\times_VT^*V} \\
	{T^*Q} && {T^*V}
	\arrow["u"', from=1-2, to=2-1]
	\arrow["v", from=1-2, to=2-3]
\end{tikzcd}\]

Abbreviate $SS(\p^*\mathcal{F})$ as $C$. Abuse notation, denote the restriction of $\Gamma_{dg}, \Gamma_{d(g\q)}$ to over $V$ by $\Gamma_{dg}, \Gamma_{d(g\q)}$ again. Let $u_*, v^*$ denote the intersection theoretic pushforward and pullback. We claim $\Gamma_{d(g\q)}=uv^{-1}\Gamma_{dg}=u_*v^*\Gamma_{dg}$ (i.e. no $> 1$ multiplicities are introduced in the intersection theoretic pull and push). Assume this for now. Note $\Gamma_{d(g\q)}$ intersects $C$ at the single point $(z, \zeta)$. We want to compute the intersection number. Since $u$ is a closed immersion and $v$ is proper smooth, $u^*C$ and $v^*\Gamma_{dg}$ also intersect at a single point and, by the projection formula from intersection theory, $C\cdot (u_*v^*\Gamma_{dg})=(u^*C)\cdot (v^*\Gamma_{dg})=(v_*u^*C)\cdot \Gamma_{dg}$. We claim that $v_*u^*C=vu^{-1}C$. Assuming this, then $vu^{-1}C^+=\q_{\circ}C^+=SS^+(R\mathcal{F})$, so $C\cdot (u_*v^*\Gamma_{dg})=SSR\mathcal{F}\cdot \Gamma_{d(g\q)}=1$.\\

It remains to show the two claims. The first claim follows from $v$ being smooth and $u$ being a closed immersion. For the second claim, we show separately below $u^*$ and $v_*$ introduce no $> 1$ multiplicities:\\

For $u^*$: This is intersecting $C$ with $Q\times_VT^*V$. Since $(a,\alpha)$ is a smooth point, $C$ is also smooth near $(z, \zeta)$. By counting dimensions, it suffices to find $n-1$ tangent vectors of $C$ which are not tangent to $Q\times_VT^*V$. One verifies that the tangents of $C$ in the $\p$ fibre direction work.\\

For $v_*$: After removing the zero section, $u^*C$ lies in the “diagonal” of $Q\times_VT^*V$, so is mapped isomorphically to its image. More precisely, consider $Q\times_{\mathbb{P}^{\vee}}T^*\mathbb{P}^{\vee}\rightarrow T^*\mathbb{P}^{\vee}$ ($Q\times_VT^*V\rightarrow T^*V$ is then its base change to $V$). Note $Q\times_{\mathbb{P}^{\vee}}(T^*\mathbb{P}^{\vee}-T_{\mathbb{P}^{\vee}}^*\mathbb{P}^{\vee})\rightarrow (T^*\mathbb{P}^{\vee}-T_{\mathbb{P}^{\vee}}^*\mathbb{P}^{\vee})$ admits a natural $\mathbb{G}_m$ action. Take the quotient (which does not change multiplicity computations), we get $Q\times_{\mathbb{P}^{\vee}}Q\rightarrow Q$ (identifying $(T^*\mathbb{P}^{\vee}-T_{\mathbb{P}^{\vee}}^*\mathbb{P}^{\vee})/\mathbb{G}_m$ with $Q$), where the map is just the projection to the second factor. Then, by $C=\p^{\circ}SS\mathcal{F}$ and the description of $T_{z'}^*Q$ in the Radon setup \ref{radonsetup}, one checks that $(u^*C-\mathrm{zero\,section})/\mathbb{G}_m$ lies in the diagonal of $Q\times_{\mathbb{P}^{\vee}}Q$, so is mapped isomorphically to its image.
\end{proof}

\subsection{The high-jet stability of $\phi$}\label{subsection on high-jet}
Return to Examples \ref{ex1}, \ref{ex2}: we noticed that when $N$ is large enough, the Swan conductors are independent of $s$ (for $s$ in a small neighbourhood of $0\in\mathbb{A}^1$), consequently $\mathrm{dim}(\phi)$'s are independent of $s$. This suggests that the dependence of vanishing cycles on the ttfun  is only up to a high enough jet. In this section, we formulate precisely the notion of vanishing cycles being stable with respect to the variation of the ttfun in order $\geq N$-terms and prove such a result in a special case. At the end, we discuss the relation between our result and a result of Saito (whose formulation and proof inspired ours).

\subsubsection{Transverse test families}\label{testfamsection}

The definition of the ttfam in the algebraic context is as follows:

\begin{definition}[transverse test family (algebraic)]\label{ttfamdef}
    A \underline{transverse test family (ttfam)} of $\mathcal{F}$ at a \emph{smooth} point $(x, \xi)$ of $SS\mathcal{F}$, denoted by $(T,U,V,f)$, is the following data (here $\mathbb{A}^1=\mathbb{A}^1_k$):\\
\[\begin{tikzcd}
	{} & {U\times T} & V & {x_T:=x\times T} & {} \\
	&& {\mathbb{A}^1_T:=\mathbb{A}^1\times T} \\
	&& {}
	\arrow[hook', from=1-3, to=1-2]
	\arrow[hook', from=1-4, to=1-3]
	\arrow["f", from=1-3, to=2-3]
\end{tikzcd}\]
where:\\
(i) T is a connected smooth finite type scheme over k. We will often identify T with $0\times T\subseteq\mathbb{A}^1\times T$.\\
(ii) U is an \emph{\etale} neighbourhood of x ($x$ is implicitly viewed as a point of $U$), V is a Zariski open of $U\times T$ containing $x_T$.\\
(iii) f is a morphism such that, for all closed points s of T, the \underline{s-slice} $f_s: V_s (:= V\times_{\mathbb{A}^1_T}\mathbb{A}^1_s) \rightarrow \mathbb{A}^1_s$ is a ttfun with respect to $\mathcal{F}$ at $(x,\xi)$ (in particular, $f_s$ is $SS\mathcal{F}$-transversal except at x).
\end{definition}

\begin{definition}[the vanishing cycle associated to a ttfam]\label{def_vanishingcycleofttfam}
    Let $\mathcal{F}\in D(X)$ and $(T,U,V,f)$ be a ttfam for $SS\mathcal{F}$ at $(x,\xi)$. \underline{The vanishing cycle associated to this ttfam} is the following sheaf on the oriented product topos $x_T\overleftarrow{\times}_{\mathbb{A}^1_T}(\mathbb{A}^1_T-T)$:
\[\phi_f(\mathcal{F}):= \Phi_f(\mathcal{F}_V)|_{x_T\overleftarrow{\times}_{\mathbb{A}^1_T}(\mathbb{A}^1_T-T)}\]
Here $\mathcal{F}_V$ is the pull back of $\mathcal{F}$ to $V$, and $\Phi$ is the vanishing cycle over general bases for $f: V\rightarrow\mathbb{A}^1_T$.\footnote{We refer to \cite{orgogozo_modifications_2006, illusie_around_2017} for details on oriented products and vanishing cycle over general bases.}\label{vanishinggeneralbase}
\end{definition}

\begin{remark}\label{rmkttfam}
    (i) In a ttfam, the condition on $f$ implies that $f$ is $SS\mathcal{F}$-transversal outside $x_T$ (essentially because for $s\hookrightarrow T$ the conormal bundle of $V_s\hookrightarrow V$ is isomorphic to the pullback of conormal bundle of $\mathbb{A}^1_s\hookrightarrow \mathbb{A}^1_T$. See, for example, \cite[the proof of 3.9.1]{saito_characteristic_2017}). This implies it is ULA with respect to $\mathcal{F}$ outside $x_T$.\\\\
    (ii) Directly from the definition of being  $SS\mathcal{F}$-transversal, $f$ is smooth in a neighbourhood of the base of $SS\mathcal{F}_V$ except possibly at points in $x_T$. If $\xi\neq 0$, then $f$ is also smooth on $x_T$; if $\xi=0$, then $f$ is not smooth on $x_T$, nevertheless it is still flat on $x_T$ because 
    $V\rightarrow \mathbb{A}^1_T$ is always a family of hypersurfaces in a neighbourhood of $x_T$.\\\\
    (iii) Apply \cite[\nopp 6.1]{orgogozo_modifications_2006} and \cite[comments after 1.6.1]{illusie_around_2017} we see: $\Phi_f(\mathcal{F}_V)$ is constructible and commutes with every base change. In particular, it is supported on $x_T\overleftarrow{\times}_{\mathbb{A}^1_T} \mathbb{A}^1_T$ and its restriction to each slice equals the usual vanishing cycle, i.e. for any closed point s of T, $\Phi_f(\mathcal{F}_V)|_{V_s\overleftarrow{\times}_{\mathbb{A}^1_T} \mathbb{A}^1_s}$ is supported on $x\overleftarrow{\times}_{\mathbb{A}^1_s} (\mathbb{A}^1_s-{0})\cong \mathbb{A}^1_{s,(0)}-\{0\}$ and canonically isomorphic to $\phi_{f_s}(\mathcal{F})_x$.\\\\
    (iv) Apply \cite[\nopp 2.8]{saito_characteristic_2017} to geometric points $x$ of $x_T$, $t$ of $T$ (identified with $0\times T\subseteq\mathbb{A}^1_T$) and $u$ of $\mathbb{A}^1_T -T$, we get a distinguished triangle:
    \[\Psi_f(\mathcal{F}_V)_{x\leftarrow t}\rightarrow \Psi_f(\mathcal{F}_V)_{x\leftarrow u}\rightarrow \Phi_f(\mathcal{F}_V)_{t\leftarrow u}\rightarrow\]
    where the first map is the cospecialisation and the second is the composition of\\ $\Psi_f(\mathcal{F}_V)_{x\leftarrow u}\rightarrow\Phi_f(\mathcal{F}_V)_{x\leftarrow u}$ and the cospecialisation $\Phi_f(\mathcal{F}_V)_{x\leftarrow u}\rightarrow\Phi_f(\mathcal{F}_V)_{t\leftarrow u}$.
    Compose this with 
\[\begin{tikzcd}
	{\mathcal{F}_x} & {\mathcal{F}_x} & 0 & {}
	\arrow[shorten <=3pt, shorten >=3pt, Rightarrow, no head, from=1-1, to=1-2]
	\arrow[shorten <=3pt, shorten >=3pt, from=1-2, to=1-3]
	\arrow[from=1-3, to=1-4]
\end{tikzcd}\]
    and take the cone (\cite[\nopp 1.1.11]{BBDG}) we get
\[\begin{tikzcd}
	{\mathcal{F}_x} & {\mathcal{F}_x} & 0 & {} \\
	{\Psi_f(\mathcal{F}_V)_{x\leftarrow t}} & {\Psi_f(\mathcal{F}_V)_{x\leftarrow u}} & {\Phi_f(\mathcal{F}_V)_{t\leftarrow u}} & {} \\
	{\Phi_f(\mathcal{F}_V)_{x\leftarrow t}} & {\Phi_f(\mathcal{F}_V)_{x\leftarrow u}} & {\Phi_f(\mathcal{F}_V)_{t\leftarrow u}} & {} \\
	{} & {} & {}
	\arrow[from=2-1, to=2-2]
	\arrow[from=2-2, to=2-3]
	\arrow[from=2-3, to=2-4]
	\arrow[shorten <=8pt, shorten >=8pt, Rightarrow, no head, from=1-1, to=1-2]
	\arrow[from=1-1, to=2-1]
	\arrow[from=1-2, to=2-2]
	\arrow[shorten <=8pt, shorten >=8pt, from=1-2, to=1-3]
	\arrow[from=1-3, to=2-3]
	\arrow[from=2-1, to=3-1]
	\arrow[from=2-2, to=3-2]
	\arrow[from=3-1, to=4-1]
	\arrow[from=3-2, to=4-2]
	\arrow[Rightarrow, no head, from=2-3, to=3-3]
	\arrow[from=3-3, to=4-3]
	\arrow[from=3-3, to=3-4]
	\arrow[from=3-1, to=3-2]
	\arrow[from=3-2, to=3-3]
	\arrow[from=1-3, to=1-4]
\end{tikzcd}\]
where the two maps in the third row are cospecialisations.\\\\
(v) Later we will often consider the condition of $\Phi_f(\mathcal{F}_V)|_{x_T\overleftarrow{\times}_{\mathbb{A}^1_T}(\mathbb{A}^1_T-T)}$ being a local system. Combine (iii) and (iv), we see: $\Phi_f(\mathcal{F}_V)|_{x_T\overleftarrow{\times}_{\mathbb{A}^1_T}(\mathbb{A}^1_T-T)}$ is a local system if and only if, in the following diagram, $(f_T, \mathcal{F}_T)$ ($\mathcal{F}_T$ is the base change of $\mathcal{F}$ to $V_T$) is ULA, which is equivalent to $\Phi_{f_T}(\mathcal{F}_T)=0$ since $\Phi$ commutes with base change by (iii) (c.f. \cite[Example 1.7 (b)]{illusie_around_2017}):
\begin{equation}\label{ttfamred}
\begin{tikzcd}
	{V_T:=V\times_{\mathbb{A}^1_T}T} \\
	T
	\arrow["{f_T}", from=1-1, to=2-1]
\end{tikzcd}
\end{equation}
 If this is satisfied, $\Phi_f(\mathcal{F}_V)$ is automatically supported on $x_T\overleftarrow{\times}_{\mathbb{A}^1_T}(\mathbb{A}^1_T-T)$.
\end{remark}

\subsubsection{Generic finite depth}\label{sectiongenfindepth}
Recall the following definition in the introduction, which makes precise the notion of the vanishing cycle being dependent on the ttfun up to the $N$-th jet.

\begin{definition}[depth of $\mathcal{F}$]
     Let $(x,\xi)$ be a smooth point of $SS\mathcal{F}$. The \underline{depth of $\mathcal{F}$ at} \underline{$(x,\xi)$} is the smallest $N\geq 2\in \mathbb{N}$ such that $\phi_f(\mathcal{F})$ is a local system for every ttfam $(T,U,V,f)$ at $(x,\xi)$ satisfying the following condition: $f_s\equiv f_{s'}$ mod $\mathfrak{m}_x^N$, for all closed points $s, s'$ of $T$. If such an $N$ does not exist, we say the depth is $\infty$.
\end{definition}

\begin{remark}\label{depthdefremark}
    (i) In the definition, one cannot allow for all functions having an isolated intersection with the $SS$ at $(x,\xi)$. Because the latter can have arbitrary intersection multiplicities\footnote{E.g. $f_s: \mathbb{A}^1\rightarrow \mathbb{A}^1$, $x\mapsto (1-s)x^n+sx^m$.}, and the depth thus defined would be infinity in general. \\
    (ii) The depth is \etale local.\\
    (iii) In the next section, we will study functorialities of the depth as well as sheaves with depth 2.
\end{remark}

We record a basic question which we do not know how to answer yet:
\begin{question}\label{questiondepthsmpull}
    How does the depth change under smooth pullbacks?
\end{question}

Here is the first version of our main result.

\begin{theorem}\label{thmstab}
    Let $X$ be a smooth surface over an algebraically closed field $k$ of characteristic $p>2$, $x\in D$ a smooth point of an irreducible divisor, $U=X-D$, $\mathcal{F}=j_{!} \mathcal{F}_U$, where  $\mathcal{F}_U$ is a tor-finite local system on $U$ concentrated in degree 0, and $(x,\xi)$ a smooth point of $SS\mathcal{F}$ with $\xi\neq 0$. Then $\mathcal{F}$ has finite depth at $(x,\xi)$ if both of the following conditions are satisfied:\\
(i) Either $(x,\xi)$ is \emph{not} conormal to $D$, or the component of $SS\mathcal{F}$ that $(x,\xi)$ lies in is the conormal of $D$.\\
(ii) Let $\overline{U}\rightarrow U$ be the minimal\footnote{I.e., the covering corresponding to the quotient $\pi_1(U,\overline{\eta}_U)\twoheadrightarrow G$, where $G$ is the image of $\pi_1(U,\overline{\eta}_U)$ in $\mathrm{Aut}_{\mathbb{Z}/\ell^n}(\mathcal{F}_{\overline{\eta}_U})$.} \etale Galois covering trivialising $\mathcal{F}$, $\overline{X}$ be the normalisation of $X$ in $\overline{U}$, and $\overline{D}=D\times_X\overline{X}$. Then $\overline{X}$ and $\overline{D}_{red}$ are smooth at points above $x$.
\end{theorem}

The proof will give an explicit bound for the depth. We will use materials from §\ref{subsec_swan} and §\ref{subsec_blowup_bound}.

\begin{proof}
    By Remark \ref{rmkttfam}.(v), $\phi_f(\mathcal{F})$ being a local system is equivalent to Diagram \ref{ttfamred} being ULA. By Deligne-Laumon (Theorem \ref{DLthm}), this is equivalent to the function $a_s$ (Formula \ref{DLeqn1}) being constant. By the constructibility of $a_s$, this is further equivalent to being constant for closed $s$. It is elementary to check that assumption (i) ensures Deligne-Laumon is applicable in our situation, and $a_s$ is just the Swan conductor of $\mathcal{F}$ restricted to the curve $C_s:=\{f_s=0\}\subseteq V_s$ at $x$. So it suffices to show:\\
    
    In the setup of the theorem, there exists some $N\in \mathbb{N}$ such that for any test curves $C\equiv C'\mod{\mathfrak{m}_x^N}$ on a same \etale open neighbourhood of $x$,\footnote{I.e., $C, C'$ arise as zero loci of ttfun's on the \etale open.} we have $\mathrm{sw}(C)=\mathrm{sw}(C')$, where $\mathrm{sw}$ denotes the Swan conductor at $x$ of the restriction of $\mathcal{F}$.\\
    
    We first fix $C\hookrightarrow W$ for some \etale open neighbourhood $W$ of $x$ and find $N_C$ such that for every other test curve $C'\hookrightarrow W$ satisfying $C\equiv C'\mod{\mathfrak{m}_x^{N_C}}$, we have $\mathrm{sw}(C)=\mathrm{sw}(C')$. Then we show $N_C$ is bounded uniformly as $C$ varies.\\
    
    Let $\overline{W}\rightarrow W$ be the base change of $\overline{X}\rightarrow X$ via $W\rightarrow X$. Consider the following diagram:
\[\begin{tikzcd}
	{\tilde{C}+E} & {\tilde{W}} \\
	{\overline{C}} & {\overline{W}} \\
	C & W
	\arrow[hook, from=3-1, to=3-2]
	\arrow[from=2-2, to=3-2]
	\arrow[from=2-1, to=3-1]
	\arrow[hook, from=2-1, to=2-2]
	\arrow["\lrcorner"{anchor=center, pos=0.125}, draw=none, from=2-1, to=3-2]
	\arrow[from=1-2, to=2-2]
	\arrow[from=1-1, to=2-1]
	\arrow[hook, from=1-1, to=1-2]
	\arrow["\lrcorner"{anchor=center, pos=0.125}, draw=none, from=1-1, to=2-2]
\end{tikzcd}\]
where $\Tilde{W}$ is obtained from $\overline{W}$ by successive blowups at closed points until $\overline{C}$ is resolved, $\Tilde{C}$ is the strict transform of $\overline{C}$, $E$ is the collection of exceptional divisors (with multiplicities). Note $\tilde{C}$ is smooth and equals to the normalisation of $\overline{C}$.\footnote{E.g., by Zariski's Main Theorem.} We require that we blowup each time simultaneously at all points above $x$ in the strict transforms of $\overline{C}$, so that the $G$-action always extends. Let $M_1=$ maximum of the multiplicities in $E$. Let $M_2=\mathrm{max}_{\sigma\neq \mathrm{id} \in G}\{ep(I_{\Tilde{\sigma},\Tilde{W}})\}$, where $\tilde{\sigma}$ is the extension (by the universal property of normalisations) of $\sigma$ to $\tilde{W}$.\\\\
\underline{Claim}: $N_C:= M_1+M_2\cdot (D\cdot C)_x\cdot |G|$ satisfies our purpose. Here $(D\cdot C)_x$ is the intersection number of $D$ and $C$ at $x$. A simple computation shows that, for $\xi$ as in (i) of Theorem \ref{thmstab}, if $\xi$ is not conormal to $D$, then $(D\cdot C)_x=1$; if $\xi$ is conormal to $D$, then $(D\cdot C)_x=2$.\\\\
Proof of the claim: if  $C'\hookrightarrow W$ is another test curve, let $\doverline{C'}$ be its normalisation. Recall the facts about the Swan conductor in §\ref{subsec_swan}. To show $\mathrm{sw}(C)=\mathrm{sw}(C')$, it suffices to show there exists a bijection of points $\{\tilde{x}\}\leftrightarrow\{\doverline{x'}\}$ of points of $\tilde{C}, \doverline{C'}$ above $x$ and for corresponding points the quantities $\lambda_{\mathcal{O}_{\tilde{C},\tilde{x}}}(\mathcal{O}_{\tilde{C},\tilde{x}}/I_{\tilde{\sigma},\tilde{C}}\mathcal{O}_{\tilde{C},\tilde{x}})$, $\lambda_{\mathcal{O}_{\doverline{C'},\doverline{x'}}}(\mathcal{O}_{\doverline{C'},\doverline{x'}}/I_{\doverline{\sigma'},\doverline{C'}}\mathcal{O}_{\doverline{C'},\doverline{x'}})$ equal for each $\sigma\neq \mathrm{id}$ in $G$. But if $C\equiv C'\mod{\mathfrak{m}_x^{N_C}}$, then $\tilde{C}\equiv \tilde{C'}\mod{I_{E_{red}}^{N_C-M_1}}$, which implies: (a) $\{\Tilde{x}\}:=\tilde{C}\cap E_{red}=\tilde{C'}\cap E_{red}$, and (b) a fortiori $\tilde{C}\equiv \tilde{C'}\mod{\mathfrak{m}_{\tilde{x}}^{N_C-M_1}}$ so $\tilde{C'}$ is also smooth (hence equal to $\doverline{C'}$). From now on we abbreviate $\mathcal{O}_{\tilde{C},\tilde{x}}$ (resp. $\mathcal{O}_{\tilde{C}',\tilde{x}}$) by $\mathcal{O}$ (resp. $\mathcal{O}'$) and drop the subscripts on “$\lambda$”. Consider $\lambda_{\mathcal{O}_{\tilde{C},\tilde{x}}}(\mathcal{O}_{\tilde{C},\tilde{x}}/I_{\tilde{\sigma},\tilde{C}}\mathcal{O}_{\tilde{C},\tilde{x}})$. We have the following estimations: $$\lambda(\mathcal{O}/I_{\tilde{\sigma},\tilde{C}}\mathcal{O})=\lambda(\mathcal{O}/I_{\Tilde{\sigma},\Tilde{W}}\mathcal{O})\leq M_2\cdot \lambda(\mathcal{O}/\sqrt{I_{\Tilde{\sigma},\Tilde{W}}}\mathcal{O})$$
$$\leq M_2\cdot \lambda(\mathcal{O}/\sqrt{I_{\Tilde{D}}}\mathcal{O})\leq M_2\cdot \lambda(\mathcal{O}/I_{\Tilde{D}}\mathcal{O})=M_2\cdot (\Tilde{D}\cdot \Tilde{C})_{\Tilde{x}}\leq M_2.(\Tilde{D}\cdot \Tilde{C})=M_2\cdot (D\cdot C)_x\cdot |G|$$ where $I_{\Tilde{D}}$ is the ideal corresponding to $\Tilde{D}$ in $\Tilde{W}$. For the last equality, we used the projection formula from intersection theory in the form of \cite[\nopp 9.2.13]{liu_algebraic_2002}. Since $\tilde{C}\equiv \tilde{C'}\mod{\mathfrak{m}_{\tilde{x}}^{M_2\cdot (D\cdot C)_x\cdot |G|}}$, $\lambda(\mathcal{O}/I_{\tilde{\sigma},\tilde{C}}\mathcal{O})$ and $\lambda(\mathcal{O}'/I_{\tilde{\sigma},\tilde{C'}}\mathcal{O}')$ must equal because $\tilde{C}$ and $\tilde{C'}$ are equal in the $(M_2\cdot (D\cdot C)_x\cdot |G|)$-th infinitesimal neighbourhood of $\tilde{x}\in \tilde{W}$. This proves the claim.\\

It remains to show $N_C=M_1+M_2\cdot (D\cdot C)_x\cdot |G|$ is bounded with respect to $C$. By Lemma \ref{blowupmult}, $M_1\leq 2^{M_C-1}\cdot \mathrm{max}_{\overline{x}}\{\mathrm{mult}_{\overline{x}}(\overline{C})\}$, where $M_C$ is the smallest number of blowup stages needed to resolve $\overline{C}$. By Lemma \ref{epcomputation}, $M_2\leq (2p+1)^{M_C}\cdot \mathrm{max}_{\sigma\neq \mathrm{id} \in G}$ $\{ep(I_{\sigma,\overline{X}})\}$. So it suffices to bound $M_C$ and $\mathrm{max}_{\overline{x}}\{\mathrm{mult}_{\overline{x}}(\overline{C})\}$. The boundedness of $M_C$ follows from Proposition \ref{prop_boundonBUstages}. We now show that $\mathrm{mult}_{\overline{x}}(\overline{C})\leq |G|\cdot (C\cdot D)_x$. Indeed, denote $\pi: \overline{X}\rightarrow X$, then $\mathrm{mult}_{\overline{x}}(\overline{C})\leq \pi_*(\overline{C}\cdot \pi^* D)=|G|\cdot (C\cdot D)_x$, where the first inequality follows from the fact that at an intersection point of two curves with no common components, the intersection number is greater than or equal to the product of their multiplicities at that point, and the second equality follows from the projection formula in \textit{loc. cit.}
\end{proof}

This completes the proof of Theorem \ref{thmstab}. The version of our main result as stated in the introduction now follows easily from this together with a few results to be proved in the next section.

\begin{theorem}[Generic finite depth]\label{thmstabbeau}
    Let $X$ be a smooth surface over an algebraically closed field $k$ of characteristic $p>2$, and $\mathcal{F}\in D(X)$. Then, there exists a Zariski open dense $V=X-$\{finitely many closed points\} and a Zariski open dense $S\subseteq SS(\mathcal{F}|_V)$ such that for any $(x,\xi)\in S$, there exists an integer $N\geq2$ such that the depth of $\mathcal{F}$ at $(x,\xi)$ is $\leq N$. Moreover, we have an upper bound: if $\mathcal{F}$ is locally constant in some punctured neighbourhood of $x$, then $N=2$; if $x$ lies in a ramification divisor $D$ of $\mathcal{F}$, then $N\leq2^{M-1}\cdot i_x\cdot |G|+(2p+1)^M\cdot \max_{\sigma\neq \mathrm{id}\in G}\{ep(I_{\sigma, \overline{X}})\}\cdot i_x\cdot |G|$, where $i_x=1$ for $(x,\xi)$ not conormal to $D$, $i_x=2$ if the component of $SS\mathcal{F}$ that $(x,\xi)$ lies in is the conormal of $D$, $M=r\cdot s\cdot i_x$, and $r, s$ are as in Proposition \ref{prop_boundonBUstages} (applied to $\overline{X}_{\hat{x}'}\rightarrow X_{\hat{x}}$ for any $x'$ above $x$).
\end{theorem}

Note $r$ and $s$ do not depend on the choice of $x'$, as the Galois action acts transitively on $\{x'\}$.

\begin{proof}
    It is clear that away from finitely many closed points and by standard dévissage (recollement and induction on amplitudes) (the dévissage works because of Lemma \ref{misclemmatriang}), we are reduced to three cases: (1) $\mathcal{F}$ is locally constant on a punctured neighbourhood of $x$; (2) $\mathcal{F}=i_*\mathcal{L}$ where $i$ is the closed immersion of a smooth curve $D$ on $X$, and $L$ is a local system on $D$; (3) the situation of Theorem \ref{thmstab}. Case (1) follows from Lemma \ref{lemma_purity} and Lemma \ref{misclemmatriang}, case (2) follows from Proposition \ref{LSmucs} and Proposition \ref{misclemmacloimm}, case (3) follows from Theorem \ref{thmstab}.
\end{proof}

\begin{example}
    Consider the sheaf in Example \ref{ex1} and $(x,\xi)=((0,0),dy)$. One checks that the normalisation of $k[x,y]$ in $k(x,y)[t]/(t^p-t-y/x^p)$ is $k[x,y,xt]/((xt)^p-x^{p-1}(xt)-y)$. Let $\tau=xt$, then $\overline{X}=\mathrm{Spec}(k[x,\tau])$. There is a single point $\overline{x}=\{x=\tau=0\}$ above $x$. One checks that $i_x=1$; $G=\mathbb{Z}/p$; $\sigma\in\mathbb{Z}/p$ acts on $\overline{X}$ by $(x,\tau)\mapsto(x,\tau+\sigma x)$, so $\forall \sigma\neq 0, I_{\sigma, \overline{X}}=(x)$ hence $ep(I_{\sigma, \overline{X}})=1$; finally, using the notations of Proposition \ref{prop_boundonBUstages}, $\Delta=-x^{p-1}$,  $\mathfrak{C}(R/A)/R=R\frac{1}{x^{p-1}}/R$, $r=p-1$, and $s=p$. So our bound gives $\mathrm{depth}(\mathcal{F})_{((0,0),dy)}\leq2^{p(p-1)-1}\cdot p+(2p+1)^{p(p-1)}\cdot p$.
    
\end{example}

We remark that, in the above example, by directly computing the Swan of test curves using explicit equations, one can show $\mathrm{depth}(\mathcal{F})=p$ at every point $((0,y),\xi)\in SS\mathcal{F}$ with $\xi\neq 0$ (see Example \ref{depthcompute} for an illustration in a simpler case). So our estimate in Theorem \ref{thmstabbeau} is not sharp.\\

We now discuss Saito's result and its relation with ours.

\begin{theorem}[{\cite[\nopp 2.14]{saito_characteristic_2015}}]\label{thmsaito}
    Let $X$ be a smooth surface over a field $k$ which is algebraically closed of characteristic $p>0$, $\Lambda$ a finite field of characteristic $\ell\neq p$, and $\mathcal{F}\in D(X,\Lambda)$ of the form $\mathcal{F}=j_{!} \mathcal{F}_U$, where $U$ is an open dense subscheme of $X$ and $\mathcal{F}_U$ is a local system on it concentrated in degree 0. Let $Z=X-U$. Let $(x,\xi)\in SS\mathcal{F}$ be a smooth point, $x$ closed. Let $f: X \rightarrow \mathbb{A}^1$ be a morphism such that $(x,\xi)$ is an isolated characteristic point of $f$ with respect to $\mathcal{F}$. Assume $f$ is flat and its restriction to $Z-x$ is \'etale. Then there exists a positive integer $N$ such that for every $g: X \rightarrow \mathbb{A}^1$ satisfying $f \equiv g \mod{\mathfrak{m}_x^N}$, there exists an isomorphism $\phi_f(\mathcal{F})_x \congto \phi_g(\mathcal{F})_x$ as objects in $D^b_c(\Lambda[G_{\eta}])$.
\end{theorem}

    (i) Saito's result fixes a test function $f$ which has an isolated characteristic point (not necessarily transverse), while our result is a \textit{uniform} bound for \textit{transverse} test functions.\\
    
    (ii) The isomorphism of vanishing cycles in Saito's result is an isomorphism of $G_{\eta}$-representations. In our result, although $\phi_f(\mathcal{F})$ being a local system certainly implies $\phi_{f_s}(\mathcal{F})$'s are isomorphic as (complexes of) vector spaces for all closed points $s$ in $T$, it is not clear what representation-theoretic data is contained in our notion of stability. On the other hand, our loss in representation-theoretic data gained us more functoriality. For example, one has a version of the 2-out-of-3 property for the depth, see Lemma \ref{misclemmatriang}.\\

We end this section with two conjectures.

\begin{conjecture}\label{finitedepth}
    Let $X$ be a smooth variety over an algebraically closed field $k$ of characteristic $p\neq 2$.Then every $\mathcal{F}\in D(X)$ has finite depth at all smooth points of $SS\mathcal{F}$.
\end{conjecture}

\begin{conjecture}\label{finitedepthrep}
    Let $X$ be a smooth variety over an algebraically closed field $k$ of characteristic $p\neq 2$, $\mathcal{F}\in D(X)$, and $(x,\xi)$ a smooth point of $SS\mathcal{F}$. Then there exists a positive integer $N$ (depending on $(x,\xi)$) such that for every \etale neighbourhood $U$ of $x$ and $f, g: U \rightarrow \mathbb{A}^1$ satisfying (a) $f$ and $g$ are ttfun's at $(x,\xi)$; (b) $f \equiv g \mod{\mathfrak{m}_x^N}$, there exists an isomorphism $\phi_f(\mathcal{F})_x \congto \phi_g(\mathcal{F})_x$ as objects in $D^b_c(\mathbb{Z}/\ell^n[G_{\eta}])$. 
\end{conjecture}

\section{$\mu c$ sheaves}\label{section-muc}
We maintain the same setup as in §\ref{stab of phi}. As mentioned in the introduction, our motivation is to build a microlocal sheaf theory in this setting. A microlocal sheaf theory “lives” on the cotangent bundle, but as discussed in §\ref{sectiononfailure}, due to the complexity of $\pi_1$ (or wild ramifications), microlocal data is huge, reflected in the fact that vanishing cycles depend on higher jets of the ttfun. This suggests at least two directions for further investigation: (i) work on a space larger than $T^*X$ (e.g., higher jet bundles), (ii) restrict the class of sheaves. The previous section is a step in (i): we showed that on a surface, generically and pointwise, the vanishing cycles “live” on some finite jet bundle. In this section, we explore the second route.\\

An immediate thought is to restrict to tame sheaves. However, this is not satisfactory, as tameness is not even preserved under the Radon transform (see Example \ref{tamenotradon}), while as mentioned in §\ref{reviewcomplex}, a fundamental feature of microlocal sheaf theory is contact invariance, of which the Radon transform is the prototypical example. Inspired by the situation in the complex analytic context, we instead consider the class of sheaves with the strongest stability. 

\begin{definition}[$\mu c$, $\mu c^s$ sheaves]\label{defmucsheaves}
    A sheaf $\mathcal{F}\in D(X)$ is \underline{$\mu c$ at a smooth point $(x,\xi)\in SS\mathcal{F}$} if for all ttfam's of $\mathcal{F}$ at $(x,\xi)$, $\phi_f(\mathcal{F})$ is a local system (Definition \ref{def_vanishingcycleofttfam}). A sheaf $\mathcal{F}$ is $\mu c$ if it is $\mu c$ at all smooth points of $SS\mathcal{F}$.\\\\
    A sheaf $\mathcal{F}\in D(X)$ is \underline{$\mu c^s$ at a smooth point $(x,\xi)\in SS\mathcal{F}$} if for all smooth morphism $p: Y\rightarrow X$ and $(y,\eta)\in T^*Y$ with $y\mapsto x, \eta=dp(\xi)$, and all ttfam's of $p^*\mathcal{F}$ at $(y,\eta)$, $\phi_f(\mathcal{F})$ is a local system. A sheaf $\mathcal{F}$ is $\mu c^s$ if it is $\mu c^s$ at all smooth points of $SS\mathcal{F}$.
\end{definition}

We record a question we do not know how to answer yet, it is a special case of Question \ref{questiondepthsmpull}.
\begin{question}
    Is $\mu c$ equivalent to $\mu c^s$?
\end{question}

A $\mu c$ sheaf is just a sheaf of depth 2 at all smooth points of its $SS$. We give them a special name as they are closest to the complex analytic case and are good candidates for microlocal constructions. Actually, we have the analogues of Propositions \ref{phistability/C}, \ref{radonstability/C}. (Note, however, we have no control on the representation structure, see item (ii) after Theorem \ref{thmsaito}.)

\begin{lemma}\label{phiindepofttfun}
    (i) Let $\mathcal{F}\in D(X)$ be $\mu c$, and $(x,\xi)$ a smooth point in $SS\mathcal{F}$. Then for every two ttfun's $f$ and $g$ at $(x, \xi)$, there exists (noncanonically) an isomorphism $\phi_f(\mathcal{F})_x\cong \phi_g(\mathcal{F})_x$ as objects in $D^b_c(\mathbb{Z}/\ell^n)$. We call this the microstalk of $\mathcal{F}$ at $(x,\xi)$.\\\\
    (ii) For $\mu c^s$ sheaves, the microstalks are invariant under the Radon transform: let $\mathcal{F}\in D(\mathbb{P})$ be $\mu c^s$, and $(x,\xi)$ be a smooth point of $SS\mathcal{F}$ with $\xi\neq 0$. Let $(a,\alpha)$ be a point in $SSR\mathcal{F}$ corresponding to $(x,\xi)$. Let $f$ (resp. $g$) be a ttfun for $\mathcal{F}$ (resp. $R\mathcal{F}$) at $(x,\xi)$ (resp. $(a,\alpha)$). Then there exists an isomorphism $\phi_g(\mathcal{F})_{a}\cong\phi_f(\mathcal{F})_{x}$ as objects in $D^b_c(\mathbb{Z}/\ell^n)$.
\end{lemma}

\begin{proof}
    (i) That $\mathcal{F}$ is $\mu c$ implies that in a ttfam $(T,U,V,g)$, the stalks of $\phi_g(\mathcal{F})$ are all isomorphic. So it suffices to show that any two ttfun's can be connected by a ttfam. Fix an \etale coordinate $\{x_1,...,x_n\}$ at $x$. Let $f$ be a ttfun on some \etale neighbourhood $U$ of $x$. The restriction of $f$ to the strict localisation $X_{(x)}\cong \mathrm{Spec}(k\{x,y\})$ is of the form $f|_{X_{(x)}}=\sum \xi_i x_i+ \sum a_{ij} x_i x_j+ H$, where $\xi_i$ are components of $\xi$ and $H$ means higher order terms. Consider the $\mathbb{A}^1=\mathrm{Spec}(k[s])$-family: $f_s:=f+(s-1)(f-(\sum \xi_i x_i+ \sum a_{ij} x_i x_j))$, then $f_s|_{X_{(x)}}=\sum \xi_i x_i+ \sum a_{ij} x_i x_j+sH$. Note these are also defined on $U$, and since we have not changed $\leq$ second order terms, $\nu$ is still a transverse intersection point of $\Gamma_{df_s}$ and $SS\mathcal{F}$. $f_s$ is a ttfun on some Zariski open neighbourhood $V_s$ of $x\in U$. Put them together, we get a ttfam $(\mathbb{A}^1,U,V,f)$, connecting $f_1=f$ to $f_0=\sum \xi_i x_i+ \sum a_{ij} x_i x_j$. Now consider $Q=$ the space of all quadratic forms $\{b_{ij}\}$ such that $\sum \xi_i x_i+ \sum b_{ij} x_i x_j$ is a ttfun on some Zariski open neighbourhood of $x\in X$. It is an open dense subspace of an affine space. Let $f_{\{b_{ij}\} }=\sum \xi_i x_i+ \sum b_{ij} x_i x_j$. This defines a ttfam $(Q,U',V',f_{\{b_{ij}\}})$ for some Zariski open $U'$ of $X$, connecting all ttfun's parametrised by $Q$.\\
    
    (ii) By Corollary \ref{Radoncompat}, $R\mathcal{F}$ is also $\mu c^s$, so its microstalks are well-defined. The same computation as in the proof of Proposition \ref{phistabradon} then gives the result.
\end{proof}

\begin{remark}\label{dimtotstabproof}
    Here is a direct proof of Proposition \ref{dimtotstab}: by the proof of Lemma \ref{phiindepofttfun}.(i), all ttfun's can be connected via ttfam's. By \cite[\nopp 1.16]{saito_characteristic_2017}, dimtot is constant in a ttfam (note that by the definition of the ttfam, the non-local-acyclicity locus is mapped isomorphically to its image in the base $\mathbb{A}^1_T$, so being flat implies being locally constant in the terminology of \cite[\nopp 1.16]{saito_characteristic_2017}).
\end{remark}

The rest of this section is devoted to:\\

\noindent (i) showing some basic sheaves are $\mu c$ and $\mu c^s$. In particular, tame simple normal crossing sheaves are $\mu c^s$;\\
(ii) showing some functorialities of the $\mu c$ and $\mu c^s$ conditions. In particular, $\mu c^s$ sheaves are preserved under the Radon transform;\\
(iii) examples.

\subsection{Basic objects}\label{subsection-basicobjects}

Recall definitions and remarks in §\ref{testfamsection}.

\begin{lemma}
    If $X$ is a smooth curve, then every $\mathcal{F}\in D(X)$ is $\mu c$.
\end{lemma}

\begin{proof}
    Let $(x,\xi)$ be a smooth point of $SS\mathcal{F}$. Notice that for any ttfam at $(x,\xi)$, in Diagram \ref{ttfamred}, $f_T$  is an isomorphism and $\mathcal{F}_T$ is a constant sheaf, so $(f_T,\mathcal{F}_T)$ is ULA.
\end{proof}

\begin{proposition}\label{LSmucs}
    Local systems are $\mu c^s$.
\end{proposition}

\begin{proof}
    It suffices to show they are $\mu c$ because pullback of local systems are local systems. The problem being \etale local, we may assume the sheaf is constant. Let $\underline{M}_X\in D(X)$ be a constant sheaf, $x\in X$.  Let $(T, U, V, f)$ be a ttfam for $\underline{M}_X$ at $(x, \xi=0)$. On each slice $V_s\xrightarrow{f_s} \mathbb{A}^1_s$, $f_s$ being a ttfun implies it has a nondegenerate quadratic singularity at $x$ over $0\in \mathbb{A}^1_s$ (in the sense of \cite[Exp. XV, 1.2.1]{SGA7}). We want to show Diagram \ref{ttfamred} is ULA. Consider the following diagram:
\[\begin{tikzcd}
	{\overline{V_T}} & {\overline{V}} & {\overline{V}-\overline{V_T}} \\
	{V_T} & V & {V-V_T} \\
	T
	\arrow["{f_T}"', from=2-1, to=3-1]
	\arrow["i", hook, from=2-1, to=2-2]
	\arrow["h", from=2-2, to=3-1]
	\arrow[from=1-1, to=2-1]
	\arrow[hook, from=1-1, to=1-2]
	\arrow["\pi"', from=1-2, to=2-2]
	\arrow["j"', hook', from=2-3, to=2-2]
	\arrow["\cong", from=1-3, to=2-3]
	\arrow["{\overline{j}}"', hook', from=1-3, to=1-2]
\end{tikzcd}\]
    where $h$ is the composition of $f$ and the projection $\mathbb{A}^1_T\rightarrow T$, $\pi: \overline{V}\rightarrow V$ is the blowup of $V$ along $x_T$. Note $\overline{V_T}\hookrightarrow \overline{V}$ is a simple normal crossing divisor over $T$. By the distinguished triangle $j_!j^*\underline{M}_V\rightarrow \underline{M}_V\rightarrow i_*i^*\underline{M}_V\rightarrow$ and the fact that $(h,\underline{M}_V)$ is ULA, to show $(f_T,i^*\underline{M}_V)$ is ULA, it suffices to show $(h,j_!j^*\underline{M}_V)$ is ULA. But $j_!j^*\underline{M}_V\cong \pi_*\overline{j}_!\underline{M}_{(\overline{V}-\overline{V_T})}$. By \cite[\nopp 4.11]{saito_characteristic_2017}, $SS(\overline{j_!}\underline{M}_{(\overline{V}-\overline{V_T})})=T_{\overline{V}}^*\overline{V}\cup T_{\overline{V}_T}^*\overline{V}$,\footnote{See Footnote \footref{footnoteonnotation} for the notation.} so $\overline{V}\rightarrow T$ is $SS(\overline{j_!}\underline{M}_{(\overline{V}-\overline{V_T})})$-transversal, so $(h\pi, \overline{j}_!\underline{M}_{(\overline{V}-\overline{V_T})})$ is ULA. By the compatibility of vanishing cycles and proper pushforwards, $\Phi_h(\pi_*\overline{j}_!\underline{M}_{(\overline{V}-\overline{V_T})})\cong\overleftarrow{\pi_*}\Phi_{h\pi}(\overline{j}_!\underline{M}_{(\overline{V}-\overline{V_T})})=0$. 
\end{proof}

\begin{proposition}\label{tamearemuc}
    Let $D\subseteq X$ be an sncd, and $j: U\rightarrow X$ its complement. If $\mathcal{F}\in D(X)$ is of the form $\mathcal{F}=j_!\mathcal{F}_U$ for $\mathcal{F}_U$ a local system tame along $D$, then $\mathcal{F}$ is $\mu c$ (hence $\mu c^s$ because its smooth pullbacks are of the same form).
\end{proposition}

Recall that in this situation $SS\mathcal{F}=T^*_XX\cup T^*_D X$ by \cite[\nopp 4.11]{saito_characteristic_2017}. We will use materials from §\ref{subsec_resolution} in this proof.

\begin{proof}
    The question being local, we may assume we are in the following situation: $X$ is affine (and smooth), $\{x_1,...,x_n\}$ is an \etale coordinate system on $X$ (i.e. $\{x_1,...,x_n\}\subseteq \mathcal{O}_X(X)$ and the map $X\xrightarrow{(x_1,...,x_n)}\mathbb{A}^n$ is \etale), and $D=\cup_{i=1}^rD_i=\cup_{i=1}^r\{x_i=0\}$ is the decomposition of $D$ into its irreducible components, $0<r\leq n$. The locally constant locus has been dealt with in the previous proposition. It suffices to show $\mathcal{F}$ is $\mu c$ at $(x,\xi)$ for $x=\mathrm{origin}$, and $\xi=dx_1+...+dx_r$.\\
    
    Let $(T,U,V,f)$ be a ttfam for $\mathcal{F}$ at $(x,\xi)$. We want to show $\Phi_{f_T}(\mathcal{F}_{T})=0$ in Diagram \ref{ttfamred}. For this, we need to understand the geometry of $D_T:=(D\times T)\cap V_T\hookrightarrow V_T$ near $x_T$. First look at each slice $D=D\times \{s\}\hookrightarrow V_s$. By Lemma \ref{lem_resolution1} and Lemma \ref
    {lem_resolution2}, we have: \\
    
    \underline{Fact}: the embedded singularity $D\cap H_s\hookrightarrow H_s$ can be resolved in two steps: first blow up at $x$, then blow up along the intersection of the exceptional divisor with the strict transform of $D_1\cap ... \cap D_{r-1}\cap H_s$. (There is only one blowup for $r=1$.)\footnote{Of course the choice $\{1,2,...,r-1\}$ is unimportant: one can choose any $r-1$ elements in $\{1,2,...,r\}$.}\\
    
    It follows that the embedded singularity $D_T\hookrightarrow V_T$ can be resolved by first blowing up along $x_T$, then blowing up along the intersection of the exceptional divisor with the strict transform of $D_{1,T}\cap ... \cap D_{r-1,T}$, where $D_{i,T}:=(D_i\times T)\cap V_T$. We get the following diagram:
\[\begin{tikzcd}
	{V_T} & {\overline{V_T}} \\
	T
	\arrow["\pi"', from=1-2, to=1-1]
	\arrow["{f_T}"', from=1-1, to=2-1]
	\arrow["{g_T}", from=1-2, to=2-1]
\end{tikzcd}\]
where $\pi$ is proper and induces an isomorphism over $V_T-D_T$, and $\pi^{-1}(D_T)\hookrightarrow \overline{V_T}$ is a sncd \emph{relative to $T$}. Note $\mathcal{F}_{T}=\pi_*\pi^*\mathcal{F}_{T}$ (because $\mathcal{F}_{T}$ is a !-extension from the open), and $\pi^*\mathcal{F}_{T}$ is still a sncd tame sheaf (by \cite[\nopp 4.4]{kerz_different_2010}). So $\Phi_{f_T}(\mathcal{F}_{T})=\overleftarrow\pi_*\Phi_{g_T}(\pi^*\mathcal{F}_{T})=0$, where the last equality comes from the fact that $SS$ of a sncd tame sheaf is conormal.
\end{proof}

\subsection{Properties}\label{subsection-properties}

\begin{proposition}[closed immersion]\label{misclemmacloimm}
    Let $i: Z\hookrightarrow X$ be a closed immersion of smooth varieties, and $\mathcal{F}\in D(Z)$. Then:\\
    (i) $\mathcal{F}$ is $\mu c$ if and only if  $i_*\mathcal{F}$ is $\mu c$. More generally, if $(z, \zeta)$ is a smooth point in $SS\mathcal{F}$ and $(x, \xi)$ is in $SS(i_*\mathcal{F})$ such that $x=i(z), di(\xi)=\zeta$, then $\mathrm{depth}(\mathcal{F})_{(z,\zeta)}=\mathrm{depth}(i_*\mathcal{F})_{(x,\xi)}$.\\
    (ii) $\mathcal{F}$ is $\mu c^s$ if and only if  $i_*\mathcal{F}$ is $\mu c^s$.
\end{proposition}

\begin{proof}
    (i) Let $(z, \zeta)$ be a point in $SS\mathcal{F}$, and $(x, \xi)$ be a point in $SS(i_*\mathcal{F})=i_{\circ}SS\mathcal{F}$ such that $x=i(z), di(\xi)=\zeta$. First note that $(x, \xi)$ is a smooth point of $SS(i_*\mathcal{F})$ if and only if $(z, \zeta)$ is a smooth point of $SS\mathcal{F}$. This follows from the observation that in 
    the following correspondence $u$ is smooth and $v$ is a closed immersion. 
\[\begin{tikzcd}
	& {Z\times_XT^*X} \\
	{T^*Z} && {T^*X}
	\arrow["u"', from=1-2, to=2-1]
	\arrow["v", from=1-2, to=2-3]
\end{tikzcd}\]

\underline{$\mathcal{F}$ $\mu c \Rightarrow i_*\mathcal{F}$ $\mu c$}: Let $(T, U, V, f)$ be a ttfam at $(x, \xi)$ for $i_*\mathcal{F}$, we want to show $\phi_f(i_*\mathcal{F})$ is locally constant. Consider the restriction of $(T,U,V,f)$ to $Z$:
\[\begin{tikzcd}
	{Z\times T} & {U_Z\times T} & {V_Z} & {z_T} \\
	{X\times T} & {U\times T} & V & {x_T} \\
	{} && {\mathbb{A}^1_T}
	\arrow[hook,"i", from=1-1, to=2-1]
	\arrow[hook, from=1-2, to=2-2]
	\arrow[from=1-2, to=1-1]
	\arrow[from=2-2, to=2-1]
	\arrow["\sim", from=1-4, to=2-4]
	\arrow["i'", from=1-3, to=2-3]
	\arrow["f", from=2-3, to=3-3]
	\arrow[hook', from=1-3, to=1-2]
	\arrow[hook', from=2-3, to=2-2]
	\arrow[hook', from=1-4, to=1-3]
	\arrow[hook', from=2-4, to=2-3]
	\arrow["\lrcorner"{anchor=center, pos=0.125, rotate=-90}, draw=none, from=1-2, to=2-1]
	\arrow["\lrcorner"{anchor=center, pos=0.125, rotate=-90}, draw=none, from=1-3, to=2-2]
\end{tikzcd}\]

By the compatibility of vanishing cycles (over general bases) with proper pushforwards, $\Phi_f((i_*\mathcal{F})_V)\cong \overleftarrow{i'}_*\Phi_{fi'}(\mathcal{F}_{V_Z})$. But $\Phi_f((i_*\mathcal{F})_V)$ is supported on $z_T\overleftarrow{\times}_{\mathbb{A}^1_T}\mathbb{A}^1_T$ and $\overleftarrow{i'}$ restricted to $z_T\overleftarrow{\times}_{\mathbb{A}^1_T}\mathbb{A}^1_T$ is an isomorphism, so $\phi_f(i_*\mathcal{F})=\Phi_f(i_*\mathcal{F}_V)|_{z_T\overleftarrow{\times}_{\mathbb{A}^1_T}(\mathbb{A}^1_T-T)}\cong \Phi_{fi'}(\mathcal{F}_{V_Z})|_{z_T\overleftarrow{\times}_{\mathbb{A}^1_T}(\mathbb{A}^1_T-T)}=\phi_{fi'}(\mathcal{F})$. So, $\mathcal{F}$ being $\mu c$, it suffices to show the restriction of $(T, U, V, f)$ to $Z$ is a ttfam for $\mathcal{F}$ at $(z, \zeta)$. In the definition of the ttfam (i), (ii) are clear. We check (iii):\\

The computation being local, we may assume $V_s=X$. Consider the correspondence above. Abbreviate $SS\mathcal{F}$ as $C$. We want to compute $C\cdot \Gamma_{d(f_s|_Z)}=C\cdot (uv^{-1}\Gamma_{df_s})$. First note $(vu^{-1}C)\cdot \Gamma_{df_s}(=(i_{\circ}SS\mathcal{F})\cdot \Gamma_{df_s}=1\cdot (x,\xi))$, $(u^{-1}C)\cdot v^{-1}\Gamma_{df_s}$, $C\cdot (uv^{-1}\Gamma_{df_s})$ are all supported at a single point because $f_s$ is a ttfun, $u$ is smooth and $v$ is a closed immersion. Then compute: $(vu^{-1}C)\cdot \Gamma_{df_s}=(v_*u^{-1}C)\cdot \Gamma_{df_s}=(u^{-1}C)\cdot v^*\Gamma_{df_s}$, where the second equality comes from $v$ being a closed immersion, third equality comes from the projection formula in intersection theory. A simple computation in a local coordinate shows that the intersection of $\Gamma_{df_s}$ and $Z\times_XT^*X$ is transverse. So $v^{-1}\Gamma_{df_s}$ is also smooth and $(u^{-1}C)\cdot v^*\Gamma_{df_s}=(u^{-1}C)\cdot v^{-1}\Gamma_{df_s}$, i.e. $u^{-1}C$ and $v^{-1}\Gamma_{df_s}$ intersect transversely at a single point. So $C$ and $uv^{-1}\Gamma_{df_s}$ also intersect transversely at a single point.\\

\underline{$\mathcal{F}$ $\mu c \Leftarrow i_*\mathcal{F}$ $\mu c$}: Let $(T, Z',V,f)$ be a ttfam at $(z,\zeta)$ for $\mathcal{F}$. If $\zeta=0$, then $\mathcal{F}$ is locally constant near $z$ and the assertion is clear. Assume $\zeta\neq 0$. By \cite[\nopp 18.1.2]{EGAIV}, we can extend $Z'$ to an \etale neighbourhood $X'$ of $x$ in $X$. By Lemma \ref{owenretraction}, after possibly shrinking $Z'$ and $X'$, there exists an \etale neighbourhood $\tilde{X'}\xrightarrow{\beta}X'$ of $x$ in $X$ and maps $\alpha, r$ satisfying the following diagram:
\[\begin{tikzcd}
	{Z'} & {\tilde{X'}} & {X'}
	\arrow["\alpha"', hook, from=1-1, to=1-2]
	\arrow["\beta"', from=1-2, to=1-3]
	\arrow["r"', curve={height=12pt}, from=1-2, to=1-1]
\end{tikzcd}\]
where $\alpha$ is a closed immersion, $\beta$ is \etale, $r$ is a retraction, and $\beta\alpha$ coincides with the closed immersion $Z'\hookrightarrow X'$. Consider the pullback of $(T, Z',V,f)$ via $r$: 
\[\begin{tikzcd}
	& {\tilde{X'}\times T} & {\tilde{V}} & {z_T} \\
	& {Z'\times T} & V & {z_T} \\
	{} && {\mathbb{A}^1_T}
	\arrow["{r\times \mathrm{id}}"', from=1-2, to=2-2]
	\arrow[Rightarrow, no head, from=1-4, to=2-4]
	\arrow[from=1-3, to=2-3]
	\arrow["f", from=2-3, to=3-3]
	\arrow[hook', from=1-3, to=1-2]
	\arrow[hook', from=2-3, to=2-2]
	\arrow[hook', from=1-4, to=1-3]
	\arrow[hook', from=2-4, to=2-3]
	\arrow["\lrcorner"{anchor=center, pos=0.125, rotate=-90}, draw=none, from=1-3, to=2-2]
	\arrow["{\tilde{f}}"'{pos=0.3}, shift left=1, curve={height=12pt}, from=1-3, to=3-3]
\end{tikzcd}\]

On each slice, $\Tilde{f_s}=f_sr$ is an extension of $f_s$. A similar intersection theoretic computation as above shows that $(T, \Tilde{X'},\Tilde{V},\Tilde{f})$ is a ttfam at $(z,(d\Tilde{f_s})_z)$ for $i_*\mathcal{F}$. Then again by the compatibility of vanishing cycles with proper pushforwards, $\phi_f(\mathcal{F})\cong\phi_{\Tilde{f}}(i_*\mathcal{F})$, and the latter is a local system by assumption.\\

The “More generally” part follows from the same method as above, plus Lemma \ref{conglemma} below.\\

(ii) “$\Rightarrow$” is clear. For “$\Leftarrow$”: The question being \etale local, we may reduce to showing that the pullback of $\mu c$ a sheaf $\mathcal{F}$ along $\mathbb{A}^m\times Z\rightarrow Z$ is $\mu c$. But it equals the restriction to $\mathbb{A}^m\times Z$ of the pullback of $\mathcal{F}$ along $\mathbb{A}^m\times X\rightarrow X$, which is $\mu c$ by (i).
\end{proof}

\begin{lemma}\label{conglemma}
    Let $f: X\rightarrow Y$ be a morphism of schemes, and $y\in Y$. If $g,h\in \mathcal{O}_{y,Y}$ are such that $f\equiv g \mod{\mathfrak{m}^N_{y,Y}}$ for some $N\in \mathbb{N}$, then $g\circ f\equiv h\circ f \mod{\mathfrak{m}^N_{x,X}}$ for every $x\in f^{-1}(y)$.
\end{lemma}

\begin{proof}
    Let $\varphi: \mathcal{O}_{y,Y}\rightarrow\mathcal{O}_{x,X}$ be the induced local ring map. $g\equiv h \mod{\mathfrak{m}^N_{y,Y}}\Rightarrow g\circ f\equiv h\circ f \mod{\varphi(\mathfrak{m}^N_{y,Y})}=\varphi(\mathfrak{m}_{y,Y})^N$, a fortiori $g\circ f\equiv h\circ f \mod{\mathfrak{m}^N_{x,X}}$.
\end{proof}

Like tame sheaves, $\mu c$ sheaves are not stable under general proper pushforwards, however they are stable under pushforwards which resemble (the pushforward part of) an integral transform. More precisely, given a map $f: Y\rightarrow X$ between smooth varieties and $\mathcal{F}\in D(Y)$, we say $f$ is \underline{special with respect to $\mathcal{F}$} if the following conditions are satisfied:\\
    (a) it is smooth and proper;\\
    (b) for any smooth point $(x, \xi)\in SS(f_*\mathcal{F})$ with $\xi\neq 0$, there exists a unique point $(y, \eta)\in (SS\mathcal{F})|_{f^{-1}(x)}$ such that $df(\xi)=\eta$. Furthermore, $(y, \eta)$ is a smooth point of $SS\mathcal{F}$;\\ 
    (c) $f_+SS\mathcal{F}=f_{\circ}SS^+\mathcal{F}$.  Recall (Radon setup \ref{radonsetup}) $SS^{+}\mathcal{F}$ denotes $SS\mathcal{F}\cup T^*_YY$. Here $f_+$ is the map from \emph{cycles} on $T^*Y$ to \emph{cycles} on $T^*X$ defined as follows: take the intersection theoretic pullback and pushforward under the correspondence $T^*Y\leftarrow Y\times_X T^*X\rightarrow T^*X$, then set the coefficient of the zero section to be $1$. In the following we will use a similar notation for pullbacks.

\begin{proposition}[special pushforward]\label{misclemmaspepush}
    Let $f: Y\rightarrow X$ be a morphism of smooth varieties, and $\mathcal{F}$ a $\mu c$ sheaf on $Y$.\\
    (i) If $f$ is special with respect to $\mathcal{F}$, then $f_*\mathcal{F}$ is $\mu c$. More generally, if $f$ is special with respect to $\mathcal{F}$, then for every pair $(x,\xi)$ and $(y,\eta)$ as above, we have $\mathrm{depth}(f_*\mathcal{F})_{(x,\xi)}\leq\mathrm{depth}(\mathcal{F})_{(y,\eta)}$.\\
    (ii) If $f$ is special with respect to $\mathcal{F}$, and $\mathcal{F}$ is $\mu c^s$, then $f_*\mathcal{F}$ is $\mu c^s$.
\end{proposition}

Note, being special implies that the pullback of every ttfun for $f_*\mathcal{F}$ at $(x,\xi)$ with $\xi\neq0$ is a ttfun for $\mathcal{F}$ at $(y,\eta)$ (c.f. the proof of Lemma \ref{radonintersect}).
\begin{proof}
    (i) Let $(x, \xi)$ be a smooth point of $SS(f_*\mathcal{F})$ with $\xi\neq 0$, $(y, \eta)$ be the point in $SS\mathcal{F}$ corresponding to it. Let $(T, U, V, g)$ be a ttfam for $f_*\mathcal{F}$ at $(x, \xi)$. Consider the pullback of this ttfam along $f$:
\[\begin{tikzcd}
	{Y\times T} & {\tilde{U}\times T} & {\tilde{V}} & {y_T} \\
	{X\times T} & {U\times T} & V & {x_T} \\
	{} && {\mathbb{A}^1_T}
	\arrow["f\times \mathrm{id}", from=1-1, to=2-1]
	\arrow[from=1-2, to=2-2]
	\arrow[from=1-2, to=1-1]
	\arrow[from=2-2, to=2-1]
	\arrow["\sim", from=1-4, to=2-4]
	\arrow[from=1-3, to=2-3]
	\arrow["g", from=2-3, to=3-3]
	\arrow[hook', from=1-3, to=1-2]
	\arrow[hook', from=2-3, to=2-2]
	\arrow[hook', from=1-4, to=1-3]
	\arrow[hook', from=2-4, to=2-3]
	\arrow["\lrcorner"{anchor=center, pos=0.125, rotate=-90}, draw=none, from=1-2, to=2-1]
	\arrow["\lrcorner"{anchor=center, pos=0.125, rotate=-90}, draw=none, from=1-3, to=2-2]
	\arrow["h"'{pos=0.3}, shift left=1, curve={height=12pt}, from=1-3, to=3-3]
\end{tikzcd}\]

As $f$ is special with respect to $\mathcal{F}$, for every closed point $s\in T$, the slice $\tilde{V}_s\xrightarrow{h_s}\mathbb{A}^1_s$ satisfies condition (iii) in the definition of a ttfam for $\mathcal{F}$ at $(y, \eta)$. Conditions (i), (ii) are clearly satisfied, so $(T, \tilde{U}, \tilde{V}, h)$ is a ttfam for $\mathcal{F}$ at $(y, \eta)$. Since $\mathcal{F}$ is $\mu c$, $\phi_h(\mathcal{F})$ is a local system. By the compatibility of vanishing cycles (over general bases) with proper pushforwards, we conclude that $\phi_g(f_*\mathcal{F})$ is also a local system.\\

The “More generally” part follows from the same method as above, plus Lemma \ref{conglemma}.\\

(ii) It suffices to check that being special with respect to a sheaf is preserved under smooth pullback. Let $g: W\rightarrow X$ be a smooth map. We have the following diagram: 
\[\begin{tikzcd}
	& Y & {Y_W} \\
	& X & W \\
	{}
	\arrow["g", from=2-3, to=2-2]
	\arrow["f"', from=1-2, to=2-2]
	\arrow["{f'}", from=1-3, to=2-3]
	\arrow["{g'}"', from=1-3, to=1-2]
	\arrow["\lrcorner"{anchor=center, pos=0.125, rotate=-90}, draw=none, from=1-3, to=2-2]
\end{tikzcd}\]

We want to show $f'$ is special with respect to $g'^*\mathcal{F}$.\\

(a): Clear;\\

(b): We need to know the intersection (away from the zero section) of $f'^{\circ}SS(f'_*g'^*\mathcal{F})=f'^{\circ}SS(g^*f_*\mathcal{F})=f'^{\circ}g^{\circ}SS(f_*\mathcal{F})=g'^{\circ}f^{\circ}SS(f_*\mathcal{F})$ and $SS(g'^*\mathcal{F})=g'^{\circ}SS\mathcal{F}$. Clearly, on the fibre $f'^{-1}(x')$, for any $x'\in W$, the intersection is none empty if and only if, on $f^{-1}(g(x'))$, the intersection of $f^{\circ}SS(f_*\mathcal{F})$ and $SS\mathcal{F}$ is nonempty, and if so the intersection is a single smooth point of $SS(g'^*\mathcal{F})$;\\

(c): $f'_+SS(g'^*\mathcal{F})=f'_+g'^+SS\mathcal{F}=g^+f_+SS\mathcal{F}=g^{\circ}f_{\circ}SS^+\mathcal{F}=f'_{\circ}g'^{\circ}SS^+\mathcal{F}=f'_{\circ}SS^+(g'^*\mathcal{F})$, where the second equality comes from the base change formula in intersection theory.
\end{proof}

Recall the notation in Radon setup \ref{radonsetup} for the next corollary and remark.

\begin{corollary}[Radon transform]\label{Radoncompat}
    The Radon transform preserves $\mu c^s$ sheaves: if $\mathcal{F}\in D(\mathbb{P})$ is $\mu c^s$, then $R\mathcal{F}\in D(\mathbb{P}^{\vee})$ is $\mu c^s$.
\end{corollary}

\begin{proof}
     It suffices to observe that, by the proof of Lemma \ref{radonintersect}, $\q$ is special with respect to $\p^*\mathcal{F}$ for any $\mathcal{F}\in D(\mathbb{P})$.
\end{proof}

\begin{remark}\label{ptwiseradoncompat}
    Note that we actually proved a pointwise statement: if $\mathcal{F}\in D(\mathbb{P})$ is $\mu c^s$ at $(x,\xi)$, $\xi\neq 0$, then $R\mathcal{F}\in D(\mathbb{P}^{\vee})$ is $\mu c^s$ at $(a,\alpha)$, for every $(a,\alpha)$ corresponding to $(x,\xi)$.
\end{remark}

\begin{lemma}[distinguished triangle]\label{misclemmatriang}
    Being $\mu c$ is compatible with distinguished triangles: let $\mathcal{F}\rightarrow \mathcal{G}\rightarrow \mathcal{H}\rightarrow$ be a distinguished triangle, assume $SS\mathcal{G}=SS\mathcal{F}\cup SS\mathcal{H}$, then:\\
    (i) $\mathcal{F}$ and $\mathcal{H}$ $\mu c$ implies $\mathcal{G}$ $\mu c$. More generally, for every smooth point $\nu\in SS\mathcal{G}$, $\mathrm{depth}(\mathcal{G})_{\nu}\leq \mathrm{max}\{\mathrm{depth}(\mathcal{F})_{\nu},$ $\mathrm{depth}(\mathcal{H})_{\nu}\}.$ \\
    (ii) $\mathcal{F}$ and $\mathcal{H}$ $\mu c
    ^s$ implies $\mathcal{G}$ $\mu c^s$.
\end{lemma}

\begin{proof}
These all follow from applying Remark \ref{rmkttfam}.(v), and using the distinguished triangle
\[\begin{tikzcd}
	{\Phi_{f_T}\mathcal{F}_T} & {\Phi_{f_T}\mathcal{G}_T} & {\Phi_{f_T}\mathcal{H}_T} & 
	\arrow[from=1-1, to=1-2]
	\arrow[from=1-2, to=1-3]
	\arrow[from=1-3, to=1-4]
\end{tikzcd}\]
Note that we need the assumption $SS\mathcal{G}=SS\mathcal{F}\cup SS\mathcal{H}$ because otherwise $SS\mathcal{F}$ and $SS\mathcal{H}$ may cancel each other and a smooth point of $SS\mathcal{G}$ may be a nonsmooth point of $SS\mathcal{F}$ or $SS\mathcal{H}$ (see, e.g., Example \ref{cancelss}).
\end{proof}

Statement (i) in the following lemma is well-known, we include a proof for completeness.

\begin{lemma}[purity]\label{lemma_purity}
    Let Z be a smooth closed subvariety of X of $\mathrm{codim} \geq 2$, and $j: U\hookrightarrow X$ the complement. Let $\mathcal{F}=j_!\mathcal{F}_U$ with $\mathcal{F}_U$ a local system. Then:\\
    (i) $SS\mathcal{F}=T^*_ZX \cup T^*_XX$.\\
    (ii) $\mathcal{F}$ is $\mu c^s$.
\end{lemma}

\begin{proof}
    (i) By induction on amplitudes and the compatibility of $\mu c$ with distinguished triangles, it suffices to deal with the case where $\mathcal{F}_U$ is concentrated in a single degree. By the Theorem of Purity (e.g. \cite[Exp. X]{SGA1}), $\mathcal{F}_U$ extends to some local system (concentrated in a single degree) $\overline{\mathcal{F}}$ on $X$. Consider the exact sequence $0\rightarrow \mathcal{F} \rightarrow \overline{\mathcal{F}} \rightarrow i_*\mathcal{F}_Z \rightarrow 0$. The second and third terms have $SS=T^*_ZX \cup T^*_XX$, so the first term has $SS\subseteq T^*_ZX \cup T^*_XX$. But $\mathcal{F}$ is not locally constant on $Z$ so $SS\mathcal{F}\neq T^*_XX$, so the inclusion is an equality by dimensional reasons.\\
    
    (ii) By induction on amplitudes, we reduce to the case $\mathcal{F}_U$ is concentrated in a single degree. Then it follows from the same exact sequence above and Lemma \ref{misclemmatriang}, 
 proposition \ref{misclemmacloimm}, and Lemma \ref{LSmucs}.
\end{proof}

\begin{corollary}\label{excodim}
    Let Z be a smooth closed subvariety of X of $\mathrm{codim} \geq 2$, $j: U\hookrightarrow X$ the complement, and $\mathcal{F}\in D(X)$. Assume $\mathcal{F}$ is \emph{not} a local system. Then:\\
    (i) $SS\mathcal{F}=T^*_ZX \cup T^*_XX$ if and only if $\mathcal{F}$ is a local system on $U$ and $Z$.\\
    (ii) If $\mathcal{F}$ satisfies the conditions in (i), then $\mathcal{F}$ is $\mu c^s$.
\end{corollary}
In real and complex analytic contexts, statement (i) is true without assumptions on the codimension (\cite[\nopp 8.4.1]{kashiwara_sheaves_1990}). In the positive characteristic algebraic context, it is false for codim = 1, see Example \ref{cancelss}.
\begin{proof}
    By the previous lemma, (ii) and “$\Leftarrow$” in (i) are clear. We show “$\Rightarrow$” in (i). Consider the distinguished triangle $j_!\mathcal{F}_U \rightarrow \mathcal{F} \rightarrow i_*\mathcal{F}_Z \rightarrow$. Note $SS(j_!\mathcal{F}_U)$ and $SS\mathcal{F}$ both equal $T^*_ZX \cup T^*_XX$, so $SS(i_*\mathcal{F}_Z)\subseteq T^*_ZX$. As  $SS(i_*\mathcal{F}_Z)=i_{\circ}SS(\mathcal{F}_Z)$, we get $SS(i_*\mathcal{F}_Z)=T^*_ZX$. It follows that $SS(\mathcal{F}_Z)=T^*_ZZ$, so $\mathcal{F}_Z$ is a local system.
\end{proof}

\subsection{Examples}\label{subsection-example}

\begin{example}\label{depthcompute}
    ($p>3$) The sheaves in Examples \ref{ex1} and \ref{ex2} are not $\mu c$ by the computations there and Lemma \ref{phiindepofttfun}. Example \ref{ex2} shows that the $SS$ being Lagrangian does not imply being $\mu c$. We do not know if being $\mu c$ implies the $SS$ being Lagrangian. Furthermore, similar computations at other points show that Example \ref{ex1} is not $\mu c$ anywhere along the divisor. On the other hand, as we show now, Example \ref{ex2} is $\mu c$ everywhere (along the smooth locus of $SS\mathcal{F}$) except above the origin.\\
    
    Consider $((0,1),dx)\in SS\mathcal{F}$, the other points are similar. Change coordinates, we may assume the sheaf is given by $t^p-t=(y+1)/x^{p-1}$ and the point in question is $(x_0,\xi)=((0,0),dx)$. By the same reasoning as in the first paragraph of the proof of Theorem \ref{thmstab}, it suffices to show: for any smooth curves on an \etale open neighbourhood of $x_0$ passing through $x_0$ with conormal at $x_0$ proportional to $dx$, $\mathrm{sw}(C)$ is independent of the curve. By the Implicit Function Theorem, any such curve is of the form $\{x=c_2y^2+c_3y^3+c_4y^4+...\}, c_2\neq0$ in the formal neighbourhood of $x_0$. The restriction of the sheaf is given by Artin-Schreier equation $$t^p-t=\frac{y+1}{(c_2y^2+c_3y^3+c_4y^4+...)^{p-1}}=\frac{y+1}{y^{2p-2}(c_2+c_3y+c_4y^2+...)^{p-1}}$$ which has Swan conductor $2p-2$, independent of $C$. 
\end{example}

In the following, we will use coordinates $[x:y:z]$ on $\mathbb{P}^2$ and $[a:b:c]$ on its dual.

\begin{example}\label{cancelss} ($p>2$, $\mathbb{Z}/\ell$-coefficient)
    Consider the Artin-Schreier sheaf on $\mathbb{P}^2$ determined by the equation $t^p-t=yz^{p-2}/x^{p-1}$, !-extended along $\{x=0\}$. Note, on the affine part $\{[x:y:1]\}$ this is just Example \ref{ex2}. As shown below, $SS\mathcal{F}=T^*_{\mathbb{P}}\mathbb{P}\cup T^*_{\{x=0\}}\mathbb{P}\cup T^*_{[0:0:1]}\mathbb{P}\cup T^*_{[0:1:0]}\mathbb{P}$. It then follows from Radon setup \ref{radonsetup}.(iii) that $SSR\mathcal{F}=T^*_{\mathbb{P}^{\vee}}\mathbb{P}^{\vee}\cup T^*_{\{b=0\}}\mathbb{P}^{\vee} \cup T^*_{\{c=0\}}\mathbb{P}^{\vee}\cup T^*_{[1:0:0]}\mathbb{P}^{\vee}$. Focus on a neighbourhood of the point $[0:1:0]\in \mathbb{P}^{\vee}$. Claim: although $SSR\mathcal{F}$ is the zero section union the conormal to a smooth divisor near this point, $R\mathcal{F}$ is \emph{not} locally constant on the divisor near this point. Indeed, as $a$ varies, the points $[a:1:0]$ correspond to the lines $\{aX+Y=0\}$ on $\mathbb{P}$ and the stalk $(R\mathcal{F})_{[a:1:0]}\cong R\Gamma(\{aX+Y=0\},\mathcal{F})$ has a jump at $a=0$ (which can be seen, for example, by computing the Euler-Poincaré characteristics using Grothendieck-Ogg-Shafarevich).\\
    
    This shows that the same statement as in Corollary \ref{excodim}.(i) for $\mathrm{codim}=1$ is false. This is in sharp contrast with the real and complex analytic cases.
\end{example}

\begin{proof}[Proof of the description for $SS\mathcal{F}$]
    \footnote{It is in general difficult to compute the $SS$. In principal, for non-degenerate sheaves one can apply \cite[\nopp 4.13]{saito_characteristic_2017}, and for rank 1 sheaves on a surface one can apply \cite{yatagawa_characteristic_2020}. In practice, it is often easier to combine these with the theorem of Deligne-Laumon (\cite[\nopp 2.1]{laumon_semi-continuite_1981}, \cite[\nopp 2.12]{saito_characteristic_2017}), the fact that the $SS$ is middle dimensional (\cite[\nopp 1.3]{beilinson_constructible_2016}), and various smooth and proper estimates to find the $SS$. This is what we do for this example.} 
    As $\mathcal{F}$ is of rank 1, by the discussion at the end of \cite[\nopp §2.1]{yatagawa_characteristic_2020}, $(SS\mathcal{F})|_{\{x=0\}}$ is the union of a single line bundle and the whole cotangent spaces at finitely many closed points. We want to show the line bundle is $\langle dx\rangle_{\{x=0\}}$ and the closed points appearing are exactly $[0:0:1]$ and $[0:1:0]$.\\
    
    First consider the affine part $\mathbb{A}^2:=[x:y:1]$. For every $\lambda\in k^{\times}$, let $\lambda^{\frac{1}{p-1}}$ be a $(p-1)$-th root. The Artin-Schreier sheaf corresponding to $t^p-t=\lambda y/x^{p-1}$ is isomorphic to that corresponding to $t^p-t=y/(\lambda^{-\frac{1}{p-1}} x)^{p-1}$. Note the former (resp. latter) is the pullback of $\mathcal{F}|_{\mathbb{A}^2}$ under $[x:y:1]\mapsto [x:\lambda y:1]$ (resp. $[x:y:1]\mapsto [\lambda^{-\frac{1}{p-1}}x:y:1]$). As the $SS$ commutes with smooth pullbacks, we see that, on $\mathbb{A}^2$, the only whole cotangent space that can appear is $T^*_{[0:0:1]}\mathbb{A}^2$.\\
    
    Consider a point $[0:\alpha:1]$, $\alpha\in k^{\times}$. At this point, the function $f=(y-\alpha)^2-x$ has differential $-dx$. We claim that $\phi_f{(\mathcal{F})}_{[0:\alpha:1]}\neq 0$, so $([0:\alpha:1],-dx)$ lies in $SS\mathcal{F}$. Indeed, apply Deligne-Laumon in the form of \cite[\nopp 2.12.1]{saito_characteristic_2017} to the map $X:=\mathbb{A}^2\rightarrow S:=\mathbb{A}^1: [x:y:1]\mapsto ((y-\alpha)^2-x)$, with $Z=\{x=0\}$, in the notation of \textit{loc. cit.} We get $\mathrm{dim}(\phi_f{(\mathcal{F})}_{[0:\alpha:1]})= \mathrm{dimtot}_{\{y=0\}}(\mathcal{F}|_{\{(y-\alpha)^2-x=0\}})- \mathrm{dimtot}_{Z_{\overline{\eta}}}(\mathcal{F}|_{X_{\overline{\eta}}})$, where $\overline{\eta}$ is a geometric point over the generic point of the strict henselisation $\mathbb{A}^1_{(0)}$. Note $Z_{\overline{\eta}}$ consists of two points, at which $\mathcal{F}$ has identical $\mathrm{dimtot}$'s. On the other hand $\mathrm{dimtot}_{\{y=0\}}(\mathcal{F}|_{\{(y-\alpha)^2-x=0\}})=1+\mathrm{sw}_{y=0}(\text{Artin-Schreier corresponding to $\frac{y}{(y-\alpha)^{2p-2}}$})=1+(2p-2)=2p-1$. So $\mathrm{dim}(\phi_f{(\mathcal{F})}_{[0:\alpha:1]})\neq 0$.\\
    
    We have shown so far that $SS\mathcal{F}$ contains the line bundle $\langle dx\rangle_{\{x=0\}}$. A similar computation using the test function $g=\frac{y}{x+1}$ (resp. $h=\frac{z}{x+1}$) shows $([0:0:1],dy)$ (resp. $([0:1:0],dz)$) lies in $SS\mathcal{F}$. Since $(SS\mathcal{F})|_{\{x=0\}}$ is the union of a single line bundle and the whole cotangent spaces at finitely many closed points, it must contain the whole cotangent spaces at $[0:0:1]$ and $[0:1:0]$. This concludes the proof.
\end{proof}

\begin{example}\label{tamenotradon}($p>2$)
    Let $Z\hookrightarrow \mathbb{P}^2$ be the closed subscheme with equation $z^{p-1}y=x^p$. Let $\mathcal{F}$ be the constant sheaf on $Z$, $*$-extended to $\mathbb{P}$. As shown below, $SS\mathcal{F}=\overline{T^*_{Z-[0:1:0]}\mathbb{P}}\cup T^*_{[0:1:0]}\mathbb{P}$, and  $SSR\mathcal{F}=T^*_{\mathbb{P}^{\vee}}\mathbb{P}^{\vee}\cup T^*_{\{b=0\}}\mathbb{P}^{\vee}\cup C$, where $C$ is the closure in $T^*\mathbb{P}^{\vee}$ of $\{([0:b:1],\langle\frac{1}{b^{1/p}}da+\frac{1}{b}db\rangle), \mathrm{all}\,\,b\}$ (for $b=0$ this means $([0:0:1],\langle db\rangle)$).\\
    
    By Remark \ref{LSmucs}, Proposition \ref{misclemmacloimm} and Remark \ref{ptwiseradoncompat}, $R\mathcal{F}$ is $\mu c^s$ except possibly along $\{b=0\}$. And $R\mathcal{F}$ has wild ramification along $\{a=0\}$ because $(SSR\mathcal{F})|_{\{a=0\}}$ is not conormal to $\{a=0\}$ generically. This shows: (a) being tame is not stable under proper pushforwards; (b) $\mu c$, $\mu c^s$ sheaves can have wildly ramifications.

\begin{proof}[Proof of the description for $SS\mathcal{F}$ and $SSR\mathcal{F}$]
    For $SS\mathcal{F}$, as $Z$ is smooth except at $[0:1:0]$, $(SS\mathcal{F})|_{\mathbb{P}-[0:1:0]}=T^*_{Z-[0:1:0]}(\mathbb{P}-[0:1:0])$. It remains to show $T^*_{[0:1:0]}\mathbb{P}$ lies in $SS\mathcal{F}$. Consider the function $f: [x:1:z]\mapsto x$. Its differential at $[0:1:0]$ is $dx$, and $\phi_f(\mathcal{F})|_{[0:1:0]}=\mathrm{cone}(\mathbb{Z}/{\ell^n}\xrightarrow{\mathrm{diagonal}} (\mathbb{Z}/{\ell^n})^{\oplus (p-1)})\neq 0$. So $([0:1:0], dx)\in SS\mathcal{F}$. As $([0:1:0], dx)\notin \overline{T^*_{Z-[0:1:0]}\mathbb{P}}$, and $SS\mathcal{F}$ is middle dimensional, the conclusion follows.\\
    
    For $SSR\mathcal{F}$, we apply Radon setup \ref{radonsetup}.(iii). The only non-immediate part is that, under Radon setup \ref{radonsetup}.(iii), $\overline{T^*_{Z-[0:1:0]}\mathbb{P}}$ corresponds to $C$. The compuation is as follows: we have $T^*_{Z-[0:1:0]}(\mathbb{P}-[0:1:0])=\{([x:x^p:1],\langle dy\rangle),\mathrm{all}\,\, x\}$. A point $[x:x^p:1]$ corresponds to the line $xa+x^pb+c=0$ in $\mathbb{P}^{\vee}$, and a point $[a:b:c]$ on this line corresponds to the codirection $\langle adx+bdy\rangle$ at $[x:x^p:1]$. Since $\langle adx+bdy\rangle=\langle dy\rangle$ if and only if $a=0$ (hence $b\neq 0$), we conclude that $\overline{T^*_{Z-[0:1:0]}\mathbb{P}}$ corresponds to the closure of $\{[0:\frac{1}{-x^p}:1],\langle xda+x^pdb \rangle,\mathrm{all\,\,nonzero}\,\, x\}$, which is $C$ above.
\end{proof}
\end{example}

\section{Appendix}
We list some analogies and contrasts among the following contexts from the microlocal perspective. These are well-known to experts. We hope they can nevertheless be useful to other readers.\\\\
(i) Ét.: bounded constructible complexes of \etale $\mathbb{Z}/\ell^n$-sheaves on smooth algebraic varieties over algebraically closed fields of positive characteristic $p\neq \ell$;\\ 
(ii) Dist.: complex valued tempered distributions on $\mathbb{R}^n$;\\
(iii) $D$-mod$_{\mathrm{h}}$: bounded holonomic complexes of algebraic $D$-modules on smooth complex algebraic varieties;\\ 
(iv) $\mathbb{C}$-ana.: bounded $\mathbb{C}$-constructible complexes of $\mathbb{C}$-sheaves on complex analytic manifolds. By Riemann-Hilbert this is equivalent to bounded regular holonomic complexes of analytic $D$-modules.\\\\
\underline{6-functor formalisms}: All except Dist. have 6-functor formalisms.
Special features:\\
Ét.: the subclass of tame sheaves is not preserved under (proper) pushforwards;\\
Dist.: having polynomial growth is not preserved under integrations;\\
$D$-mod$_{\mathrm{h}}$: subclass of regular holonomic $D$-modules is  stable under 6-functors.\\\\
\underline{The singular support ($SS$) and characteristic cycle ($CC)$:} the $SS$ and $CC$ are defined for Ét., $D$-mod$_{\mathrm{h}}$ and $\mathbb{C}$-ana. The $CC$ satisfies index formulas. The $SS$ is also defined for Dist. (which are called wavefront sets). The $SS$'s are closed conical subsets in $T^*X$. Special features:\\
Ét.: the $SS$'s are middle dimensional;\\
Dist.: no special feature;\\
$D$-mod$_{\mathrm{h}}$: the $SS$'s are Lagrangian; for general coherent (not necessarily holonomic) $D$-modules the $SS$'s are coisotropic;\\
$\mathbb{C}$-ana.: the $SS$'s are Lagrangian.\\\\
\underline{Fourier transforms:} All of them have Fourier transforms (on $X=\mathbb{A}^n$). Special features:\\
Ét., Dist., $D$-mod$_{\mathrm{h}}$: equivalence on the whole category;\\
$\mathbb{C}$-ana.: not an equivalence on the whole category but becomes an equivalence after restriction to conic sheaves.\\\\
\underline{Microlocal data:}\\
Ét.: large data contained in wild ramifications (in dimension one: representation of local Galois groups);\\
Dist.: large data contained in (essential) singularities\footnote{E.g., Great Picard's Theorem: at an essential singularity $x$ of a complex analytic function $f$, in every punctured neighbourhood of $x$, $f$ takes all complex values infinitely many times, with at most one exception.};\\
$D$-mod$_{\mathrm{h}}$: large data contained in irregular singularities (in dimension one: Stokes data); for general analytic $D$-modules, microlocalisation can be carried out and is the content of the theory of algebraic analysis (microfunctions, microdifferential operators...);\\
$\mathbb{C}$-ana.: relatively small data, microlocalisation can be carried out.\\\\
\underline{Extension properties:}\\
Ét.: fix $\mathbb{A}^1_{k,(0)}-\{0\}\rightarrow\mathbb{G}_{m,k}\hookrightarrow\mathbb{A}^1_k\hookrightarrow\mathbb{P}^1_k$. Given a local system on $\mathbb{A}^1_{k,(0)}-\{0\}$, there are in general many ways to extend it to a local system on $\mathbb{G}_{m,k}$. However, there exists a unique extension (up to isomophism) which is \underline{special} (\cite{katz_local--global_1986});\\
Dist.: given a smooth function on a small punctured disk at the origin of $\mathbb{R}$, there are many ways to extend to a smooth function on $\mathbb{R}-\{0\}$;\\
$D$-mod$_{\mathrm{h}}$: fix $D^\circ\rightarrow\mathbb{G}_{m,\mathbb{C}}\hookrightarrow\mathbb{A}^1_\mathbb{C}\hookrightarrow\mathbb{P}^1_\mathbb{C}$, where $D^\circ$ is the punctured formal disk at the origin. Given a vector bundle with a flat connection on $D^\circ$, there are in general many ways to extend it to a vector bundle with a flat connection on $\mathbb{G}_{m,\mathbb{C}}$. However, there exists a unique extension (up to isomophism) which is \underline{special} (\cite[\nopp II.2.4]{katz_calculation_1987});\\
$\mathbb{C}$-ana.: there is a unique way to extend a local system on a punctured small disk at the origin of $\mathbb{C}$ to a local system on $\mathbb{C}^{\times}$.
\newpage
\printbibliography
\Addresses

\end{document}